\pdfoutput=1
\RequirePackage{ifpdf}
\ifpdf 
\documentclass[pdftex]{sigma}
\else
\documentclass{sigma}
\fi

\numberwithin{equation}{section}

\newtheorem{Theorem}{Theorem}[section]

\newtheorem{Lemma}[Theorem]{Lemma}
\newtheorem{Proposition}[Theorem]{Proposition}
 { \theoremstyle{definition}

\newtheorem{Remark}[Theorem]{Remark} }

\usepackage{eucal,bbold,bbm}

\makeatletter
\def\amsbb{\use@mathgroup \M@U \symAMSb}
\makeatother

\newcommand{\dd}{\operatorname{d}}
\newcommand{\e}{{\mbox{\rm e}}}
\newcommand{\mb}[1]{{\boldsymbol{#1}}}
\newcommand{\mc}[1]{{\mathcal{#1}}}
\newcommand{\mr}[1]{{\mathrm{#1}}}
\newcommand{\got}[1]{{\mathfrak{#1}}}
\newcommand{\db}[1]{{\amsbb{#1}}}
\newcommand{\pa}{\partial}
\newcommand{\un}{{\mathbbm{1}}_n}
\newcommand{\R}{{\amsbb{R}}}
\newcommand{\C}{{\amsbb{C}}}
\newcommand{\N}{{\amsbb{N}}}

\newcommand{\Hi}{{\got{H}}}

\def\ii{{\rm i}}
\newcommand{\tr}{\operatorname{Tr}}
\newcommand{\h}{{\got{h}}}

\newcommand{\Hinf}{{\mathcal{H}^{\infty}}}

\renewcommand{\P}{{\amsbb{P}}}
\newcommand{\Phinf}{{\P (\Hinf )}}

\newcommand{\zn}{{\mathbb{0}}_n}

\newcommand{\gl}{{\mc{L}}}

\begin{document}

\allowdisplaybreaks

\newcommand{\arXivNumber}{1512.00601}

\renewcommand{\PaperNumber}{064}

\FirstPageHeading

\ShortArticleName{Balanced Metric and Berezin Quantization on the Siegel--Jacobi Ball}

\ArticleName{Balanced Metric and Berezin Quantization\\ on the Siegel--Jacobi Ball}

\Author{Stefan BERCEANU}
\AuthorNameForHeading{S.~Berceanu}
\Address{National Institute for Physics and Nuclear Engineering, Department of Theoretical Physics,\\
 PO BOX MG-6, Bucharest-Magurele, Romania}
\Email{\href{mailto: Berceanu@theory.nipne.ro}{Berceanu@theory.nipne.ro}}
\URLaddress{\url{http://www.theory.nipne.ro/index.php/mcp-home}}

\ArticleDates{Received March 03, 2016, in f\/inal form June 17, 2016; Published online June 27, 2016}

\Abstract{We determine the matrix of the balanced metric of the Siegel--Jacobi ball and its inverse. We calculate the scalar curvature, the Ricci form and the Laplace--Beltrami operator of this manifold. We discuss several geometric aspects related with Berezin quantization on the Siegel--Jacobi ball.}

\Keywords{Jacobi group; Siegel--Jacobi ball; balanced metric; homogenous K\"ahler manifolds; Laplace--Beltrami operator; scalar curvature; Ricci
form; Berezin quantization}

\Classification{32Q15; 81S10; 53D50; 57Q35; 81R30}

\section{Introduction}\label{intro}

The Jacobi groups are semidirect products of appropriate semisimple real algebraic groups of Hermitian type with Heisenberg groups \cite{bs,ez,LEE03,TA99}. As unimodular, nonreductive, algebraic groups of Harish-Chandra type~\cite{SA80}, the Jacobi groups are intensively studied in mathema\-tics~\cite{gem, bb,TA90,TA92,TA99,Y02,Y08,Y10}. The Siegel--Jacobi (partially bounded~\cite{Y08,Y10}) domains are homogeneous manifolds associated to the Jacobi groups by the generalized Harish-Chandra embedding. It was underlined~\cite{csg,SB14,berr} that the Siegel--Jacobi disk is a~nonsymmetric, reductive~\cite{nomizu}, quantizable~\cite{Cah}, Q-K.~Lu K\"ahler manifold~\cite{lu66}, not Einstein with respect to the balanced metric~\cite{arr, don}, with constant negative scalar curvature. In the present paper we investigate similar geometric properties for the Siegel--Jacobi ball.

As was already emphasized \cite{jac1,sbj,nou}, the Jacobi group is relevant in several branches of physics, as: quantum mechanics, geometric quantization, quantum optics, nuclear structure, signal processing~\cite{BERC08B, KRSAR82,Q90,SH03}. In particular, the Jacobi group is responsible~\cite{BERC09} for the squeezed states in quantum optics~\cite{mandel}. Recently, new signif\/icant applications of the Jacobi group have been underlined in various domains as: quantum tomography~\cite{marmo}, quantum teleportation~\cite{chi}, Vlasov kinetic equation~\cite{gbt}, general symmetry methods for analyzing solutions of dif\/ferential equations using Lie's prolongation method~\cite{hunz, sep}, establishing the link between the standard Gaussian distribution and the Siegel--Jacobi space~\cite{mol}. Berezin's quantization and Berezin symbols related to the Jacobi group have been also investigated~\cite{ca1,ca2}.

The real Jacobi group of index $n$ is def\/ined as $G^J_n(\R )=\mr{H}_n(\R)\rtimes\operatorname{Sp}(n,\R) $, where $\mr{H}_n(\R)$ is the real Heisenberg group of real dimension $(2n+1)$ \cite{nou, Y02}. Let $g=(M,X,k)$, $g'=(M',X',k')\in G^J_n(\R )$, where $X=(\lambda,\mu)\in\R^{2n}$ and $(X,k)\in \mr{H}_n(\R)$. Then the composition law in $G^J_n(\R )$ is
\begin{gather*}
gg'=\big(MM',XM'+ X', k+k'+XM'JX'^t\big),
\end{gather*}
where
\begin{gather*} J=\left(\begin{matrix} \zn & \un\\-\un &
 \zn \end{matrix}\right) .\end{gather*}
Also it is considered the restricted real Jacobi group $G^J_n(\R )_0$, consisting only of elements of the form above, but $g=(M,X)$.

To the Jacobi group $G^J_n(\R )$ it is associated the homogeneous manifold~-- the Siegel--Jacobi upper half-plane~-- $\mc{X}^J_n\approx\C^n \times
\mc{X}_n$, where the Siegel upper half-plane $\mc{X}_n= \mr{Sp}(n,\R)/\mr{U}(n)$ is realized as
\begin{gather*}\mc{X}_n:=\{V\in M(n,\C)\,| \,V=S+\ii R, \, S,
R\in M(n,\R),\, R>0, \,
 S^t=S, \, R^t=R\} . \end{gather*}

The (complex version of the) Jacobi group of index $n$ is def\/ined as $G^J_n=\mr{H}_n\rtimes\operatorname{Sp}(n,\R)_{\C}$, where $\mr{H}_n$ denotes the
$(2n+1)$-dimensional Heisenberg group \cite{sbj,nou, Y02}. The composition law is
\begin{gather*}
(g_1,\alpha_1,t_1)(g_2,\alpha_2, t_2)= \big(g_1 g_2, g_2^{-1}\times \alpha_1+\alpha_2, t_1+ t_2 +\Im \big(g^{-1}_2\times\alpha_1\bar{\alpha}_2\big)\big),
\end{gather*}
where $\alpha_i \in\C^n$, $t_i\in\R$ and $g_i\in\operatorname{Sp}(n,\R)_{\C}$, $i=1,2$ have the form
\begin{gather}\label{dgM}
g = \left( \begin{matrix}p & q\\ \bar{q} & \bar{p}\end{matrix}\right), \qquad p,q\in M(n,\C),
\end{gather}
 and $g\times\alpha = p \alpha + q \bar{\alpha} $, and $g^{-1}\times\alpha ={p}^*\alpha -q^t\bar{\alpha}$.

The Siegel--Jacobi ball, denoted $\mc{D}^J_n$ \cite{sbj}, is the homogeneous manifold associated with the Jacobi group $G^J_n $, whose points are in $\C^n\times\mc{D}_n$, where $\mc{D}_n$ denotes the Siegel (open) ball. The non-compact hermitian symmetric space $ \operatorname{Sp}(n,\R)_{\C}/\operatorname{U}(n)$ admits a matrix realization as a bounded homogeneous domain:
\begin{gather}\label{dn}
\mc{D}_n:=\big\{W\in M (n, \C )\colon W=W^t,\, N>0,\, N:=\un-W\bar{W} \big\}.
\end{gather}
$\mc{D}_n$ is a hermitian symmetric space of type CI (cf.~\cite[Table~V, p.~518]{helg}), identif\/ied with the symmetric bounded domain of type~II, $\got{R}_{\text{II}}$ in Hua's notation~\cite{hua}.

In \cite{JGSP,sbj,nou} we have attached coherent states (CS) \cite{perG} to the Jacobi group $G^J_n$ with support on the Siegel--Jacobi ball $\mc{D}^J_n$. The particular case of coherent states attached to the Jacobi group $G^J_1$ def\/ined on the Siegel--Jacobi disk $\mc{D}^J_1$ has been investigated in~\cite{jac1,csg}. In the present paper we don't use ef\/fectively the coherent states, but we use results obtained in geometry using the coherent state approach. The homogeneous K\"ahler two-form on~$\mc{D}^J_n$, denoted $\omega_{\mc{D}^J_n}$, has been obtained in \cite{sbj,nou} from the scalar product of two
coherent state vectors, and this will be the starting point of the present investigation. The reproducing kernel function for the Siegel--Jacobi
ball was obtained previously, see \cite[p.~532]{neeb} and references there. We recall that in~\cite{jac1} we have shown that when expressed in variables on~$\mc{X}_1$ obtained (via the partial Cayley transform) from variables on $\mc{D}^J_1$, the K\"ahler two-form~$\omega_{\mc{X}^J_1}$ is the one determined by
Berndt~\cite{bern,bs} and K\"ahler~\cite{cal3,cal}. Applying the partial Cayley transform, which we recall later in Proposition~\ref{THETAS}, from $\omega_{\mc{X}^J_n}$ obtained in~\cite{Y07}, it was obtained the homogeneous K\"ahler two-form on the Siegel--Jacobi ball~$\omega_{\mc{D}^J_n} $~-- in fact, Yang considers a Jacobi group more general than~$G^J_n$. In~\cite{nou} we have shown that when expressed in appropriate variables, $\omega_{\mc{D}^J_n}$ is a particular case of that obtained by Yang in~\cite{Y10}. In~\cite{berr} we have underlined that the homogeneous metric corresponding to~$\omega_{\mc{D}^J_1}$ is a balanced metric~\cite{arr, don}. In the present paper it is emphasized that the metric corresponding to $\omega_{\mc{D}^J_n}$~\cite{JGSP,sbj,nou}
is the balanced metric. We also point out several geometric aspects related with Berezin quantization on the Siegel--Jacobi ball.

The paper is laid out as follows. In Section~\ref{PRL1} we recall several notions which are used in the paper. The notion of balanced metric on homogeneous K\"ahler manifolds is brief\/ly recalled in Section~\ref{PRL2}. Also other notions which appear in connection with Berezin quantization, as quantizable manifold, Berezin kernel, CS-type group, Q.-K.~Lu manifold and diastasis are brief\/ly recalled. The homogenous metric on the Siegel upper half plane and Siegel ball, which we need later, are recalled in Section~\ref{PRL3}. We have considered useful to collect in Section~\ref{PRL4} several formulas referring to the dif\/ferential geometry of the Siegel--Jacobi disk~$\mc{D}^J_1$, a guide for the formulas deduced in the present paper for the Siegel--Jacobi ball~$\mc{D}^J_n$. Section~\ref{GMM} contains the original results of the present paper. The balanced metric is obtained from the K\"ahler potential, calculated previously \cite{sbj} using the coherent states. Do to the symmetry of the variables $W\in M(n,\C)$ describing points on the Siegel ball $\mc{D}_n$, we write down the matrix associated with the metric $h_{\mc{D}^J_n}(z,W)$ on the Siegel--Jacobi ball as a four-blocks matrix expressed in $(z,W)\in\C^n\times\mc{D}_n$, but also in a variable $\eta$ related with $(z,W)$ by $\eta=(\un-W\bar{W})^{-1}(z+W\bar{z})$. The signif\/icance of
the change of coordinates ${\rm FC}\colon (\eta,W) \rightarrow (z,W)$ as a homogeneous K\"ahler dif\/feomorphism was underlined in~\cite{nou} in the context of the {\it fundamental conjecture}~\cite{DN} for homogeneous K\"ahler manifolds (Gindikin and Vinberg~\cite{GV}). In order to determine the inverse of the metric matrix attached to the metric~$h_{\mc{D}^J_n}$, we need the ``inverse'' of the metric matrix~$\h^k_{\mc{D}_n}$ of the Siegel ball, presented in Section~\ref{PRL3}~-- the complication comparatively with that on the noncompact Grassmann manifold~\cite{gras} comes from the fact the matrices describing the points on Siegel ball are symmetric~-- see Lemma \ref{CUL}. $(h^k_{\mc{D}_n})^{-1}$ is calculated in Proposition~\ref{LLB},
where more details on the calculation of the Laplace--Beltrami operator on Siegel ball are given. Theorem~\ref{mainTH} contains the matrix of the metric of the Siegel--Jacobi ball and its ``inverse'' (in the sense of Lemma~\ref{CUL}). Proposition~\ref{geoPL} collects a~description of the geometry of the Siegel--Jacobi ball from the point of view of Berezin's quantization. In Section~\ref{SCR} it is shown that the scalar curvature of the Siegel--Jacobi ball is constant and negative. This is compatible with Theorem~4.1 in~\cite{loi04} which asserts that a~K\"ahler manifold that admits a regular quantization has constant scalar curvature. It is also shown that $\mc{D}^J_n$ is not an Einstein manifold with respect to the balanced metric, but it is one with respect to the Bergman metric. Section~\ref{LBO} contains the Laplace--Beltrami ope\-ra\-tor on~$\mc{D}^J_n$. Remark~\ref{equivv} 
essentially asserts that if the metric on the homogeneous space in $M = G/H$ is invariant to the action of~$G$ on~$M$, then the same is true for the Laplace--Beltrami operator, a fact used in the proof of Theorem~\ref{mainTH}. In Appendix~\ref{APP1} we reproduce also a~lemma which calculates the invariant volume of the Siegel--Jacobi ball used in~\cite{sbj}. In Appendix~\ref{APP2} already mentioned are presented auxiliary calculations to the paper~\cite{maa} of Maass and a~remark concerning the results obtained in~\cite{hua59} referring to the Laplace--Beltrami operator on the Siegel ball and Siegel upper half plane.

The main new results of the present paper are contained in Remark~\ref{kra}, Lemma~\ref{CUL}, Theo\-rem~\ref{mainTH}, Propositions~\ref{geoPL},~\ref{mainPR} and~\ref{incaun}. They generalize to the Siegel--Jacobi ball the results recalled in Section~\ref{PRL4} obtained for the Siegel--Jacobi disk.

\textbf{Notation.} We denote by $\amsbb{R}$, $\amsbb{C}$, $\amsbb{Z}$, and $\amsbb{N}$ the f\/ield of real numbers, the f\/ield of complex numbers, the ring of integers, and the set of non-negative integers, respectively. We denote the imaginary unit $\sqrt{-1}$ by~$\ii$, and the real and imaginary part of a complex number by $\Re$ and respectively $\Im$, i.e., we have for $z\in\C$, $z=\Re z+\ii \Im z$, and $\bar{z}= cc(z)= \Re z-\ii \Im z$. \mbox{$M(m\times n, \amsbb{F})\approxeq\amsbb{F}^{mn}$} denotes the set of all $m\times n $ matrices with entries in the f\/ield~$\amsbb{F}$. $M(n\times 1,\amsbb{F})$ is
identif\/ied with~$\amsbb{F}^n$. Set $M(n,\amsbb{F}):=M(n\times n,\amsbb{F})$. For any $A\in M_{n}(\amsbb{F})$, $A^{t}$ denotes the transpose matrix of~$A$. For $A\in M_{n}(\amsbb{C})$, $\bar{A}$ denotes the conjugate matrix of $A$ and $A^{*}=\bar{A}^{t}$. For $A\in M_n(\amsbb{C})$, the inequality $A>0$ means that $A$ is positive def\/inite. The identity matrix of degree $n$ is denoted by~$\un$ and~$\zn$ denotes the $M(n,\amsbb{F})$-matrix with all entries $0$. We consider a complex separable Hilbert space~$\got{H}$ endowed with a scalar product which is antilinear in the f\/irst argument, $(\lambda x,y)=\bar{\lambda}(x,y)$, $x,y\in\got{H}$, $\lambda\in\C\setminus 0$. If $A$ is a linear operator, we denote by~$A^{\dagger}$ its adjoint. If~$\pi$ is representation of a Lie group~$G$ on the Hilbert space $\Hi$~and $\got{g}$ is the Lie algebra of~$G$, we denote $\mb{X}:=\dd \pi(X)$ for $X\in\got{g}$. A complex analytic manifold is {\it K\"ahlerian} if it is endowed with a Hermitian metric whose imaginary part $\omega$ has $\dd \omega = 0$ \cite{helg}. We denote the action of a Lie group $G$ on the space~$M$ by $G\times M\rightarrow M$. A coset space $M=G/H$ is {\it homogenous K\"ahlerian} if it caries a K\"ahlerian structure invariant under the group~$G$~\cite{bo}. By a~{\it K\"ahler homogeneous diffeomorphism} we mean a~dif\/feomorphism $\phi\colon M\rightarrow N$ of homogeneous K\"ahler manifolds such that $\phi^*\omega_N=\omega_M$. We use Einstein convention that repeated indices are implicitly summed over. If the K\"ahler-two form has the local expression
\begin{gather}\label{kall}
\omega_M(z)=\ii\sum_{\alpha,\beta=1}^n h_{\alpha\bar{\beta}} (z) \dd z_{\alpha}\wedge
\dd\bar{z}_{\beta}, \qquad h_{\alpha\bar{\beta}}= \bar{h}_{\beta\bar{\alpha}}= h_{\bar{\beta}\alpha},
\end{gather}
 we denote
\begin{gather}\label{corectt}
h^{\alpha\bar{\beta}}:=(h_{\alpha\bar{\beta}})^{-1},
\end{gather}
i.e., we have
\begin{gather}\label{sum2}
h_{\alpha\bar{\epsilon}}h^{\epsilon \bar{\beta}}=\delta_{\alpha\beta}.
\end{gather}
In this paper we use the following expression for the Laplace--Beltrami operator on K\"ahler manifolds $M$ with the K\"ahler two-form~\eqref{kall}:
\begin{gather}\label{LPB}
\Delta_M(z):=\sum_{\alpha,\beta=1}^n(h_{\alpha\bar{\beta}})^{-1}\frac{\pa^2}{\pa \bar{z}_{\alpha}\pa{{z}_{\beta}}},
\end{gather}
cf., e.g., Lemma~3 in the Appendix of~\cite{ber74} or in \cite[equation~(5.2.15), p.~253]{jost}, where the Laplace--Beltrami operator is $-\Delta_M(z)$ and
the author uses the convention $h^{i\bar{j}}h_{k\bar{j}}=\delta_{ij}$ (see \cite[equation~(5.2.14), p.~252]{jost}) instead of our notation~\eqref{corectt},~\eqref{sum2}. This means that $h^{l\bar{k}}$ in Jost's notation corresponds to $h^{k\bar{l}}=(h_{k\bar{l}})^{-1}$ in our notation.

\section{Preliminaries}\label{PRL1}
\subsection{Balanced metric and Berezin quantization via coherent states}\label{PRL2}

We consider a $G$-invariant K\"ahler two-form $\omega_M$ \eqref{kall} on the $2n$-dimensional homogeneous mani\-fold $M=G/H$ derived from the K\"ahler potential $f(z,\bar{z})$ \cite{chern}
\begin{gather}\label{Ter}
h_{\alpha\bar{\beta}}= \frac{\pa^2 f}{\pa {z}_{\alpha}\pa \bar{z}_{{\beta}}} .
\end{gather}

As was underlined in \cite{berr} for $\mc{D}^J_1$, the homogeneous hermitian metric determined in \cite{jac1,sbj,nou} is in fact a balanced metric, because it corresponds to the K\"ahler potential calculated as the product of two CS-vectors
\begin{gather}\label{FK}
f(z,\bar{z})=\ln K_M(z,\bar{z}), \qquad K_M(z,\bar{z})=(e_{\bar{z}},e_{\bar{z}}).
\end{gather}
In the approach of Perelomov \cite{perG} to CS, it is supposed that there exists a continuous, unitary, irreducible representation~$\pi$ of a Lie group $G$
 on the separable complex Hilbert space~$\Hi$. {\em The coherent vector mapping} $\varphi$ is def\/ined locally, on a coordinate neighborhood $\mc{V}_0\subset M$ (cf.~\cite{sb6, last}):
 \begin{gather*}
 \varphi \colon \ M\rightarrow \bar{\Hi}, \qquad \varphi(z)=e_{\bar{z}},
\end{gather*}
where $ \bar{\Hi}$ denotes the Hilbert space conjugate to~$\Hi$. The vectors $e_{\bar{z}}\in\bar{\Hi}$ indexed by the points $z \in M $ are called {\it
Perelomov's CS vectors}. Using Perelomov's CS vectors, we consider Berezin's approach to quantization on K\"ahler manifolds with the supercomplete set of vectors verifying the Parceval overcompletness identity \cite{ber73,ber74,berezin,ber75}
\begin{gather}\label{PAR}
(\psi_1,\psi_2)_{\mc{F}_{K}}=\int_M (\psi_1,e_{\bar{z}}) (e_{\bar{z}},\psi_2) \dd \nu_M(z,\bar{z}), \qquad \psi_1,\psi_2\in\Hi ,\\
\dd{\nu}_{M}(z,\bar{z})=\frac{\Omega_M(z,\bar{z})}{(e_{\bar{z}},e_{\bar{z}})},
\qquad \Omega_M:=\frac {1}{n!} \underbrace{\omega\wedge\cdots\wedge\omega}_{\text{$n$ times}}.\nonumber
\end{gather}
The reproducing kernel for the Hilbert space of holomorphic, square integrable functions with respect to the scalar product of the type \eqref{PAR} was calculated via CS as the scalar product $K_M(z,\bar{w})=(e_{\bar{z}},e_{\bar{w}})$ \cite{jac1,sbj,nou}.

On the other side, let us consider the weighted Hilbert space $\Hi_f$ of square integrable holomorphic functions on~$M$, with weight $\e^{-f}$ \cite{eng}
\begin{gather}\label{HIF}
\Hi_f=\left\{\phi\in\operatorname{hol}(M) \,| \int_M\e^{-f}|\phi|^2 \Omega_M <\infty \right\}.
\end{gather}

In order to identify the Hilbert space $\Hi_f$ def\/ined by \eqref{HIF} with the Hilbert ${\mc{F}_{K}}$ space with scalar product~\eqref{PAR}, we have to consider the $\epsilon$-function \cite{cah+,Cah,raw}
\begin{gather*}
\epsilon(z) = \e^{-f(z)}K_M(z,\bar{z}).
\end{gather*}
If the K\"ahler metric on the complex manifold $M$ is obtained from the K\"ahler potential via \eqref{kall},~\eqref{Ter} and~\eqref{FK} is such that $\epsilon(z)$ is a positive constant, then the metric is called {\it balanced}. This denomination was f\/irstly used in~\cite{don} for compact manifolds, then
it was used in \cite{arr} for noncompact manifolds and also in~\cite{alo} in the context of Berezin quantization on homogeneous bounded domains.

The {\it balanced hermitian metric} of $M$ in local coordinates is
\begin{gather}\label{herm}
\dd s^2_M(z,\bar{z}) =\sum_{\alpha,\beta=1}^n \frac{\pa^2}{\pa z_{\alpha} \pa\bar{z}_{\beta}} \ln (K_M(z,\bar{z})) \dd z_{\alpha}\otimes \dd\bar{z}_{\beta} ,
\qquad K_M(z,\bar{z})=(e_{\bar{z}},e_{\bar{z}}).
\end{gather}

We recall that in \cite{cah+,Cah,raw} Berezin's quantization via coherent states was globalized and extended to non-homogeneous manifolds in the context of geometric (pre-)quantiza\-tion~\cite{Kos}. To the K\"ahler manifold $(M,\omega)$, it is also attached the triple $\sigma =(\gl,h,\nabla)$, where~$\gl$ is a~holomorphic (prequantum) line bundle on $M$, $h$ is the hermitian metric on~$\gl$ and~$\nabla$ is a connection compatible with metric and the
K\"ahler structure \cite{SBS}. The manifold is called {\it quantizable} if the curvature of the connection~\cite{chern} $F(X,Y)=\nabla_X\nabla_Y-\nabla_Y\nabla_X-\nabla_{[X,Y]}$ has the property that $F=-\ii \omega_M $, or $\pa\bar{\pa} \log \hat{h} =\ii \omega_M$,
where $\hat{h}$ is a local representative of $h$. Then $\omega_M$ is integral, i.e., $c_1[\mc{L}]=[\omega_M]$. The reproducing (weighted Bergman) kernel admits the series expansion
\begin{gather}\label{funck}
K_M(z,\bar{w})\equiv (e_{\bar{z}},e_{\bar{w}}) = \sum_{i=0}^{\infty}\varphi_i(z)\bar{\varphi}_i(w), \end{gather}
where $\Phi= (\varphi_0,\varphi_1,\dots )$ is an orthonormal base with respect to the scalar product \eqref{PAR}:
\begin{gather*}
\int_M\bar{\varphi}_i\varphi_j\frac{\Omega_M}{K_M}=\delta_{ij}, \qquad i,j\in \N.
\end{gather*}
The base $\Phi$ is f\/inite-dimensional for compact manifolds~$M$.

Let $\xi\colon \Hi\setminus 0\rightarrow\db{P}(\Hi)$ be the canonical projection $\xi(\mb{z})=[\mb{z}]$. The {\it Fubini--Study metric} in the
nonhomogeneous coordinates~$[z]$ is the hermitian metric on~$\db{CP}^{\infty}$ (see \cite{koba} for details)
\begin{gather}\label{FBST}
\dd s^2|_{\rm FS}(\mb[{z}])= \frac{(\dd\mb{z},\dd\mb{z}) (\mb{z},\mb{z})-(\dd\mb{z},\mb{z}) (\mb{z},\dd\mb{z})}{(\mb{z},\mb{z})^2}.
\end{gather}

It was proved \cite{raw} (see also \cite{berr}) that:
\begin{Remark}\label{asaofi}If $\epsilon(z)$ is constant on $M$, then the balanced Hermitian metric on $M$ is the pullback
\begin{gather}\label{KOL}
\dd s^2_M(z)=\iota_M^*\dd s^2_{\rm FS}(z)= \dd s^2_{\rm FS}(\iota_M(z))
\end{gather}
 of the Fubini--Study metric \eqref{FBST} via the embedding
\begin{gather}\label{invers}
\iota_M\colon \ M\hookrightarrow \db{CP}^{\infty},\qquad \iota_M(z) = [\varphi_0(z):\varphi_1(z):\dots ] .
\end{gather}
If $M$ is a compact manifold, the embedding \eqref{invers} is the Kodaira embedding.
\end{Remark}

If the homogeneous K\"ahler manifold $M=G/H$ to which we associate the Hilbert space of functions $\mc{F}_K$ with respect to the scalar product~\eqref{PAR} admits a holomorphic embedding $\iota_M \colon M \hookrightarrow \Phinf$, then $M$ is called a~CS-{\em orbit}, and $G$ is called a CS-{\it type group}~\cite{sb6, lis,neeb}.

We denote the {\it normalized Bergman kernel} ({\it the two-point function} of $M$ \cite{ber97,SBS}) by
\begin{gather}\label{kmic}
\kappa_M(z,\bar{z}'):=\frac{K_M(z,\bar{z}')}{\sqrt{K_M(z)K_M(z')}}=
(\tilde{e}_{\bar{z}},\tilde{e}_{\bar{z}'})=\frac{(e_{\bar{z}},e_{\bar{z}'})}{\|e_{\bar{z}}\|\|e_{\bar{z}'}\|} .
\end{gather}
The set $\Sigma_z:=\{ z'\in M \,|\,\kappa_M(z,\bar{z}')=0\} $ was called \cite{ber97,SBS} {\it polar divisor} relative to $z\in M$, while a manifold for which
$\Sigma_z=\varnothing$, $\forall\, z\in M$ was called in~\cite{berr} a {\it Q.-K.~Lu manifold}, extending to manifolds a denomination introduced for domains in~$\C^n$~\cite{lu66}. Note that for a particular class of compact homogeneous manifolds that includes the hermitian symmetric spaces, $\Sigma_z$~is equal with the {\it cut locus} relative to $z\in M$ (see the def\/inition of the cut locus, e.g.,~\cite[p.~100]{koba2}), and~$\Sigma_z$ is a {\it divisor} in the sense of algebraic geometry \cite{ber97,SBS}.

By {\it Berezin kernel} $b_M\colon M\times M\rightarrow [0,1]\in \R$, in this paper we mean:
\begin{gather*}
b_M(z,z'):=|\kappa_M(z,\bar{z}')|^2.
\end{gather*}
The {\it Calabi's diastasis} \cite{calabi1} expressed via the CS-vectors \cite{cah+} reads
\begin{gather}\label{DIA}
D_M(z,z'):= -\ln b_M(z,z') = -2\ln \left\vert(\tilde{e}_{\bar{z}},\tilde{e}_{\bar{z}'})\right\vert .
\end{gather}

Generalizing a theorem proved for homogeneous bounded domains~\cite{alo}, in~\cite{LM15} Loi and Mossa have proved:
\begin{Theorem}\label{LMN}
Let $(M,\omega)$ be a simply-connected homogeneous K\"ahler manifold such that the associated K\"ahler form $\omega$ is integral. Then there exists a constant $\mu_0>0$ such that $M$ equipped with $\mu_0\omega$ is projectively induced.
\end{Theorem}
With Theorem \ref{LMN} and the recalled def\/initions, we can formulate the
\begin{Remark}\label{kra}
Let $M=G/H$ be a simply-connected homogeneous K\"ahler manifold. Then the following assertions are equivalent:
\begin{enumerate}\itemsep=0pt
\item[A)] $M$ is a quantizable K\"ahler manifold,
\item[B)] $M$ admits a balanced metric,
\item[C)] $M$ is CS-type manifold and $G$ is a CS-type group,
\item[D)] $M$ is projectively induced and we have \eqref{KOL}, \eqref{invers}.
\end{enumerate}
\end{Remark}
 Loi and Mossa have proved \cite{LM15} the following
\begin{Theorem}\label{LMM} Let $(M,\omega)$ be a homogeneous K\"ahler manifold. Then the following are equivalent:
\begin{enumerate}\itemsep=0pt
\item[a)] $M$ is contractible,
\item[b)] $(M,\omega)$ admits a global K\"ahler potential,
\item[c)] $(M,\omega)$ admits a global diastasis $D_M\colon M\times M\rightarrow \R$,
\item[d)] $(M,\omega)$ admits a Berezin quantization.
\end{enumerate}
\end{Theorem}

The proof of the Theorem \ref{LMM} is based on the fundamental conjecture for homogeneous bounded domains~\cite{DN} and the asymptotic of the Bergman kernel~\cite{eng}.

Theorems \ref{LMN}, \ref{LMM} and Remark~\ref{kra} will be used in the proof of Proposition~\ref{geoPL}.

\subsection{Elements of geometry of the Siegel upper half-plane and Siegel ball}\label{PRL3}

Firstly we recall some standard facts about the symplectic group and its algebra, see references and details in~\cite{sbj,nou}. Even if we will be concerned mainly with the Jacobi group $G^J_n$, we recall some basic facts about Cayley and partial Cayley transform which will appear in the paper.

\subsubsection{The Cayley transform}

If $X \in\got{sp}(n,\R)$, i.e., $X^tJ+JX=0$, then
\begin{gather*}
X=\left( \begin{matrix} a & b \\c & -a^t\end{matrix}\right),
\qquad\text{where} \quad b=b^t,\quad c=c^t, \quad a,b,c\in M(n,\R ).
\end{gather*}

If $ g=\left(\begin{matrix}a & b\\c& d\end{matrix}\right) \in {{\operatorname{Sp}}}(n,\R)$, i.e., $g^tJg=J$, then the matrices $a,b,c,d\in M(n,\R)$ have the properties
\begin{gather*}
a^tc = c^ta, \!\!\qquad b^t d = d^t b, \!\!\qquad a^td - c^t b =\un ,\!\! \qquad 
ab^t = ba^t, \!\!\qquad cd^t=dc^t, \!\!\qquad ad^t-bc^t =\un . 
\end{gather*}

We have the correspondence
\begin{gather}\label{pana}
g=\left(\begin{matrix}a & b\\c & d\end{matrix}\right) \in M(2n,\R
)\leftrightarrow g_{\C} = \mc{C}^{-1}g\mc{C} = \left(\begin{matrix} p & q \\ \bar{q}
&\bar{p} \end{matrix}\right), \qquad p, q\in M(n,\C ),
\end{gather}
where \begin{gather}
\mc{C}=\left(\begin{matrix} \ii \un & \ii \un
\\
-\un & \un\end{matrix}\right), \qquad \mc{C}^{-1}= \frac{1}{2}
\left(\begin{array}{cc} - \ii \un & -\un \\ -\ii \un &
 \un\end{array}\right), \nonumber\\
2 a = p+q +\bar{p}+\bar{q} ,\qquad
2 b = \ii (\bar{p}-\bar{q}-p+q), \qquad 2 c = \ii (p+q-\bar{p}-\bar{q}), \nonumber\\
 2 d = p-q +\bar{p}-\bar{q}, \qquad 2p = a+d+\ii (b-c), \qquad 2q= a-d -\ii (b+c). \label{CUCURUCU}
\end{gather}
 To every $g\in \operatorname{Sp}(n,\R )$, we associate via \eqref{pana} $g\mapsto g_{\C}\in \operatorname{Sp}(n,\R)_{\C}\equiv\operatorname{Sp}(n, \C)\cap {\rm U}(n,n)$ as in~\eqref{dgM}, where the matrices $p,q\in M(n,\C)$ have the properties
\begin{gather}\label{simplectic}
 pp^*- qq^* = \un,\qquad pq^t=qp^t, \qquad 
 p^*p-q^t\bar{q} = \un ,\qquad p^t\bar{q}=q^*p .
 \end{gather}

Let us consider an element $h=(g,l)$ in $G^J_n(\R )_0$, i.e.,
\begin{gather}\label{mM}
g=\left(\begin{matrix} a & b\\ c & d\end{matrix}\right)\in{\operatorname{Sp}}(n,\R),
\qquad l=(n,m)\in\R^{2n}, \end{gather}
and $V\in\mc{X}_n$, $u\in\C^n\equiv\R^{2n}$.

Let $g\in \operatorname{Sp}(n,\R )_{\C}$ be of the form \eqref{dgM}, \eqref{simplectic} and $\alpha, z\in\C^n$. The (transitive) action $(g,\alpha)\times(W,z)=(W_1,z_1)$ of the Jacobi group $G^J_n$ on the Siegel--Jacobi ball $\mc{D}^J_n$ is given by the formulas~\cite{sbj}
\begin{gather}
W_1 =(pW+q)(\bar{q}W+\bar{p})^{-1}=(Wq^*+p^*)^{-1}(q^t+Wp^t),\nonumber\\
z_1 = (Wq^*+ p^*)^{-1}(z+ \alpha -W\bar{\alpha}).
\label{TOIU}
\end{gather}

Now we consider the partial Cayley transform \cite{sbj,nou} $\Phi\colon \mc{X}^J_n\rightarrow \mc{D}_n^J$, $\Phi(V,u)=(W,z)$
\begin{gather}\label{bigtransf}
W = (V-\ii\un )(V+\ii\un )^{-1}, \qquad 
 z = 2\ii (V+\ii\un)^{-1} u,
\end{gather}
with the inverse partial Cayley transform $\Phi^{-1}\colon \mc{D}^J_n\rightarrow \mc{X}^J_n$, $\Phi^{-1}(W,z)=(V,u)$
\begin{gather} V = \ii (\un-W ) ^{-1} (\un+W ),\qquad
 u = (\un-W)^{-1}z. 
 \label{big11}
\end{gather}
Let us def\/ine $\Theta\colon G^J_n(\R)_0\rightarrow (G^J_n)_0$, $\Theta(h)=h_*$, $h=(g,n,m)$, $h_*=(g_{\C},\alpha)$.

We have proved in \cite[Proposition 2]{nou} (see also \cite{gem, Y08})
\begin{Proposition}\label{THETAS}
$\Theta$ is an group isomorphism and the action of $(G^J_n)_0$ on $\mc{D}^J_n$ is compatible with the action of $G^J_n(\R )_0$ on $\mc{X}^J_n$ through the biholomorphic partial Cayley transform~\eqref{bigtransf}, i.e., if $\Theta(h)=h_*$, then $\Phi h=h_*\Phi$. More exactly, if the action of $(G^J_n)_0$ on~$\mc{D}^J_n$ is given by~\eqref{TOIU}, then the action of $G^J_n(\R)_0$ on~$\mc{X}^J_n$ is given by $(g,l)\times(V,u)\rightarrow (V_1,u_1)\in\mc{X}^J_n$, where
\begin{gather}
V_1 = (aV+b)(cV+d)^{-1} =(Vc^t+d^t)^{-1}(Va^t+b^t), \nonumber\\
 u_1 = (Vc^t+d^t)^{-1}(u+Vn+m).\label{conf}
\end{gather}
The matrices $g$ in~\eqref{mM} and $g_{\C}$ in~\eqref{dgM} are related by
\eqref{pana}, \eqref{CUCURUCU}, while $\alpha=m+\ii n$, $m,n\in\R^n$.
\end{Proposition}

\subsubsection{The metric}
Siegel has determined the metric on $\mc{X}_n$,
$\operatorname{Sp}(n,\R)$-invariant to the action \eqref{conf} (see
 \cite[equation~(2)]{sieg} or \cite[Theorem~3, p.~644]{hua44}):
\begin{gather}\label{msig}
\dd s^2_{\mc{X}_n}(R)=\tr \big(R^{-1}\dd V R^{-1} \dd \bar{V}\big).
\end{gather}
With the Cayley transform \eqref{big11} and the relations
\begin{gather}\label{scg1}
\dot{V} = 2\ii U \dot{W} U, \qquad U=(\un-W)^{-1},\qquad
2 R =(\un +W)U+(\un+\bar{W})\bar{U} ,
\end{gather}
 introduced in \eqref{msig}, it is obtained the metric on $\mc{D}_n$, $\operatorname{Sp}(n,\R)_{\C}$-invariant to the action~\eqref{TOIU}:
\begin{gather}\label{mtrball}
\dd s^2_{\mc{D}_n}(W)= 4\tr(M\dd W\bar{M}\dd \bar{W}),\qquad W\in\mc{D}_n, \qquad M=(\un-W\bar{W})^{-1},
\end{gather}
associated with the K\"ahler two-form \eqref{omd}, modulo the factor $\frac{2}{k}$.
\begin{Remark}The metric \eqref{mtrball} can be written as in~\eqref{NEMM},~\eqref{neM}.
\end{Remark}
\begin{proof}
Equation \eqref{mtrball} can be written down as
\begin{gather*}
\tfrac{1}{4}\dd s^2_{\mc{D}_n}(W) = \sum_{p,q,m,n}H_{pq\bar{m}\bar{n}}\dd w_{pq}\dd \bar{w}_{mn},\\
H_{pq\bar{m}\bar{n}}:=M_{np}M_{mq}, \qquad M:=(\un- W\bar{W})^{-1}.
\end{gather*}
We have
\begin{gather}
 \tfrac{1}{4}\dd
s^2_{\mc{D}_n}(W) =\sum_{p<q;\, m<n}(H_{pq\bar{m}\bar{n}}+H_{pq\bar{n}\bar{m}}+H_{qp\bar{m}\bar{n}}+H_{qp\bar{n}\bar{m}})\dd w_{pq}\dd
\bar{w}_{mn} \nonumber\\
\hphantom{\tfrac{1}{4}\dd s^2_{\mc{D}_n}(W) =}{} + \sum_{p<q}(H_{pq\bar{m}\bar{m}}+H_{qp\bar{m}\bar{m}}) \dd w_{pq}\dd\bar{w}_{mm}\nonumber\\
\hphantom{\tfrac{1}{4}\dd s^2_{\mc{D}_n}(W) =}{} +\sum_{m<n}( H_{pp\bar{m}\bar{n}}+H_{pp\bar{n}\bar{m}})\dd w_{pp}\dd \bar{w}_{mn}+ \sum H_{pp\bar{m}\bar{m}}\dd w_{pp}\dd \bar{w}_{mm}.\label{SPS}
\end{gather}
For reasons which will be clarif\/ied further (see formulas \eqref{620Ec}, \eqref{hc23}), we write \eqref{SPS} as
\begin{gather}\label{NEMM}
\tfrac{1}{4}\dd
s^2_{\mc{D}_n}(W)= \sum_{p\le q; \, m\le n}h^k_{pqmn}\dd w_{pq}\dd
\bar{w}_{mn}, \end{gather}
where
\begin{gather}\label{neM}
h^k_{pq\bar{m}\bar{n}} = 2M_{mp}M_{nq}(1-\delta_{pq})+2M_{mq}M_{np}(1-\delta_{mn})+
M^2_{mp}\delta_{pq}\delta_{mn},
\end{gather}
which proves the statement of the remark.
\end{proof}

In order to take into account the symmetry of the matrix $W=(w_{ij})_{i,j=1,\dots,n}$, we introduce the notation (see \cite[equation (4.5)]{nou}):
\begin{gather}\label{mircea}
\Delta^{ij}_{pq}:=\frac{\pa w_{ij}}{\pa w_{pq}} = \delta_{ip}\delta_{jq}+\delta_{iq}\delta_{jp}-\delta_{ij}\delta_{pq}\delta_{ip},
\qquad w_{ij}=w_{ji} .
\end{gather}

For calculation of the ``inverse'' $h^{-1}$ of the metric matrix $h$ of $\mc{D}^J_n$~-- see Theorem \ref{mainTH}~-- we need the ``inverse'' of the matrix~\eqref{neM}, which we calculate in Proposition~\ref{LLB}, exploiting the formula given in~\cite{maa} for the Laplace--Beltrami operator~\eqref{LPB} on~$\mc{X}_n$ with the result (see~\eqref{ofi}):
\begin{gather}\label{kpqmn}
 k_{mn\bar{u}\bar{v}}=\tfrac{1}{2}(N_{vn}\bar{N}_{mu}+N_{vm}\bar{N}_{nu}), \qquad N=\un-W\bar{W}.\end{gather}

We calculate
\begin{gather}\label{BIGM}
E_{pq}^{uv}:=\sum_{m\leq n}h^k_{pq\bar{m}\bar{n}}k_{mn\bar{u}\bar{v}}.
\end{gather}

We shall prove a relation which gives the sense of the ``inverse''
matrix of \eqref{neM}, important in Theorem \ref{mainTH}:
\begin{Lemma}\label{CUL} The ``inverse''
matrix of the metric matrix \eqref{neM} on the Siegel ball $\mc{D}_n$
is the matrix~\eqref{kpqmn} and we have
\begin{gather}\label{culmea}
 E^{uv}_{pq}= \Delta^{uv}_{pq}.\end{gather}\end{Lemma}
\begin{proof}
a) Firstly, we consider the case $p=q$ in \eqref{BIGM}. We get successively
\begin{gather}
E^{uv}_{pp} = \sum_{m\le n}(2-\delta_{mn})(\un - W\bar{W})^{-1}_{mp}(\un - \bar{W}W)^{-1}_{pn}k_{mn\bar{u}\bar{v}}\nonumber\\
\hphantom{E^{uv}_{pp}}{} = (\un - W\bar{W})^{-1}_{mp}(\un - \bar{W}W)^{-1}_{pm}(\un-W\bar{W})_{vm}(\un - \bar{W}W)_{mu}\nonumber\\
\hphantom{E^{uv}_{pp}=}{}+ \sum_{m<n}(\un-W\bar{W})^{-1}_{mp}(\un-\bar{W}W)^{-1}_{pn} \nonumber\\
\hphantom{E^{uv}_{pp}=}{} \times [(\un-W\bar{W})_{vn}(\un-\bar{W}W)_{mu}+(\un-W\bar{W})_{vm}(\un-\bar{W}W)_{nu}]\nonumber\\
\hphantom{E^{uv}_{pp}}{} =(\un-W\bar{W})^{-1}_{mp}(\un-\bar{W}W)^{-1}_{pm}(\un-W\bar{W})_{vm}(\un-\bar{W}W)_{mu}\nonumber\\
\hphantom{E^{uv}_{pp}=}{} +
\sum_{m<n}(\un-W\bar{W})^{-1}_{np}(\un-\bar{W}W)^{-1}_{pm}(\un-W\bar{W})_{vn}(\un-\bar{W}W)_{mu}\nonumber\\
\hphantom{E^{uv}_{pp}=}{} +
 \sum_{n<m}(\un-W\bar{W})^{-1}_{mp}(\un-\bar{W}W)^{-1}_{pm}(\un-W\bar{W})_{vn}(\un-\bar{W}W)_{mu}\nonumber\\
\hphantom{E^{uv}_{pp}}{} =\delta_{up}\delta_{vp}.\label{PSLIT1}
\end{gather}
b) Now we consider the case $p\not=q$ in \eqref{BIGM}. We get successively
\begin{gather}
E^{uv}_{pq} = 2\!\sum_{m\le n}\!\! \big[(\un - W\bar{W})^{-1}_{mp}(\un - \bar{W}W)^{-1}_{qn}\! +
(\un - W\bar{W})^{-1}_{mq}(\un - \bar{W}W)^{-1}_{pn}(1 - \delta_{mn})\big]k_{mn\bar{u}\bar{v}}\!\nonumber\\
\hphantom{E^{uv}_{pq}}{}
=2(\un-W\bar{W})^{-1}_{mp}(\un-\bar{W}W)^{-1}_{qm}(\un-W\bar{W})_{vm}(\un-\bar{W}W)_{mu}\nonumber\\
\hphantom{E^{uv}_{pq}=}{}+\sum_{m<n}\big[(\un-W\bar{W})^{-1}_{mp}(\un-\bar{W}W)^{-1}_{qn}+(\un-W\bar{W})^{-1}_{mq}(\un-\bar{W}W)^{-1}_{pn}\big] \nonumber\\
\hphantom{E^{uv}_{pq}=}{} \times\big[(\un - W\bar{W})_{vn}(\un - \bar{W}W)_{vn}(\un - \bar{W}W)_{mu} + (\un - W\bar{W})_{vm}(\un - \bar{W}W)_{nu}\big]\nonumber\\
\hphantom{E^{uv}_{pq}}{}= \delta_{pu}\delta_{vq}+\delta_{pv}\delta_{qu}.\label{PSLIT2}
\end{gather}
Putting together \eqref{PSLIT1} and \eqref{PSLIT2}, with formula
\begin{gather*}E^{pq}_{uv}=E^{pq}_{uv}(1-\delta_{pq})+E^{pq}_{uv}\delta_{pq},\end{gather*} we
have proved~\eqref{culmea}.
\end{proof}

\subsection{Geometry of the Siegel--Jacobi disk}\label{PRL4}
We recall some formulas describing the dif\/ferential geometry of the Siegel--Jacobi disk $\mc{D}^J_1$ \cite{jac1,csg,berr}.

The (transitive) action of the group $ G^J_1=\mr{H}_1\rtimes\text{SU}(1,1)\ni (g,\alpha)\times(z,w)\rightarrow(z_1,w_1)\in \mc{D}^J_1$ on the Siegel disk is given by the formulas
\begin{gather}\label{dg} w_1
=\frac{a w+ b}{\delta}, \qquad \delta=\bar{b}w+\bar{a},\qquad
\text{SU}(1,1) \ni g= \left( \begin{matrix}a & b\\ \bar{b} &
\bar{a}\end{matrix}\right),
\end{gather}
where $|a|^2-|b|^2=1$,
\begin{gather}\label{xxx}
z_1=\frac{\gamma}{\delta}, \qquad \gamma = z +\alpha-\bar{\alpha}w.
\end{gather}
The balanced K\"ahler two-form is obtained with formulas \eqref{kall}, \eqref{Ter}
\begin{gather*}
-\ii \omega_{k\mu}(z,w) = h_{z\bar{z}}\dd z\wedge \dd\bar{z}+h_{z\bar{w}}\dd z\wedge
\dd \bar{w} -h_{\bar{z}w}\dd\bar{z}\wedge \dd w +h_{w\bar{w}}\dd w\wedge \dd\bar{w} ,
\end{gather*}
where the K\"ahler potential \eqref{FK} is
\begin{gather}\label{keler}
f(z,w) =\mu\frac{2z\bar{z}+z^2\bar{w}+\bar{z}^2w}{2P} -2k\ln (P), \qquad P=1-w\bar{w}.
\end{gather}
$k$ indexes the positive discrete series of $\text{SU}(1,1)$ ($2k\in\N$), while $\mu>0$ indexes the representations of the Heisenberg group.
\begin{Proposition}\label{prop2}
The balanced K\"ahler two-form $\omega_{k\mu}$ on $\mc{D}^J_1$, $G^J_1$-invariant to the action \eqref{dg}, \eqref{xxx}, can be written as
\begin{gather*}
-\ii\omega_{k\mu}(z,w) =2k\frac{\dd w \wedge \dd\bar{w}}{P^2} + \mu\frac{\mc{A}\wedge \bar{\mc{A}}}{P},
\qquad \mc{A}=\dd z+\bar{\eta}(z,w)\dd w,
\\ 
z=\eta-w\bar{\eta},\qquad\text{and} \qquad \eta = \eta(z,w):= \frac{z+\bar{z}w}{1-w\bar{w}}.
\end{gather*}
\end{Proposition}

The matrix of the balanced metric \eqref{herm} $h=h(\varsigma)$, $\varsigma:=(z,w)\in\C\times\mc{D}_1$, determined with the K\"ahler potential \eqref{keler}, reads
\begin{gather}\label{metrica}
h(\varsigma):=\left(\begin{matrix}h_{z\bar{z}} & h_{z\bar{w}}\\
 h_{w\bar{z}}=\bar{h}_{z\bar{w}} & h_{w\bar{w}}\end{matrix}\right) =\left(\begin{matrix} \frac{\mu}{P} & \mu \frac{\eta}{P} \\
\mu\frac{\bar{\eta}}{P} &
\frac{2k}{P^2}+\mu\frac{|\eta|^2}{P}\end{matrix}\right).
\end{gather}

The inverse of the matrix \eqref{metrica} reads
\begin{gather}\label{hinv}
h^{-1}(\varsigma)= \left(\begin{matrix}h^{z\bar{z}}&
 h^{z\bar{w}}\\h^{w\bar{z}}&h^{w\bar{w}}\end{matrix}\right) = \left(\begin{matrix}
 \frac{P}{\mu}+\frac{P^2|\eta|^2}{2k} & -\frac{P^2\eta}{2k} \\
-\frac{P^2\bar{\eta}}{2k} & \frac{P^2}{2k}\end{matrix}\right).
\end{gather}

If we introduce the notation
\begin{gather*}
\mc{G}_M(z):=\det
(h_{\alpha\bar{\beta}})_{\alpha,\beta=1,\dots,n},\end{gather*}
then we f\/ind
\begin{gather}\label{g311}
\mc{G}_{\mc{D}^J_1}(z,w)=\frac{2k\mu}{(1-w\bar{w})^3},\qquad z\in\C, \qquad |w|<1.
\end{gather}

\section{Elements of geometry of the Siegel--Jacobi ball}\label{GMM}

\subsection{The balanced metric}\label{BLM}

The Jacobi algebra is the semi-direct sum $\got{g}^J_n:= \got{h}_n\rtimes \got{sp}(n,\R )_{\C}$ \cite{JGSP,sbj,nou}. The Heisenberg
algebra $\got{h}_n$ is generated by the boson creation (respectively, annihilation) operators~${a}_i^{\dagger}$~(${a}_i$),~$i,j =1,\dots,n$, which verify the canonical commutation relations
\begin{gather}\label{baza1M}
\big[a_i,a^{\dagger}_j\big]=\delta_{ij}, \qquad [a_i,a_j] = \big[a_i^{\dagger},a_j^{\dagger}\big]= 0 .
\end{gather}
$\got{h}_n$ is an ideal in $\got{g}^J_n$, i.e., $[\got{h}_n,\got{g}^J_n]=\got{h}_n$, determined by the commutation relations:
\begin{subequations}\label{baza3M}
\begin{gather}
\label{baza31}\big[a^{\dagger}_k,K^+_{ij}\big] = [a_k,K^-_{ij}]=0, \\
 [a_i,K^+_{kj}] = \tfrac{1}{2}\delta_{ik}a^{\dagger}_j+\tfrac{1}{2}\delta_{ij}a^{\dagger}_k ,\qquad
 \big[K^-_{kj},a^{\dagger}_i\big] = \tfrac{1}{2}\delta_{ik}a_j+\tfrac{1}{2}\delta_{ij}a_k , \\
 \big[K^0_{ij},a^{\dagger}_k\big] = \tfrac{1}{2}\delta_{jk}a^{\dagger}_i,\qquad
\big[a_k,K^0_{ij}\big]= \tfrac{1}{2}\delta_{ik}a_{j} .
\end{gather}
\end{subequations}
 $K^{\pm,0}_{ij}$ are the generators of the $\got{sp}(\R)_{\C}$ algebra:
\begin{subequations}\label{baza2M}
\begin{gather}
 [K_{ij}^-,K_{kl}^-] = [K_{ij}^+,K_{kl}^+]=0 , \qquad 2\big[K^-_{ij},K^0_{kl}\big] = K_{il}^-\delta_{kj}+K^-_{jl}\delta_{ki}\label{baza23}, \\
 2[K_{ij}^-,K_{kl}^+] = K^0_{kj}\delta_{li}+
K^0_{lj}\delta_{ki}+K^0_{ki}\delta_{lj}+K^0_{li}\delta_{kj}
\label{baza22}, \\
2\big[K^+_{ij},K^0_{kl}\big] = -K^+_{ik}\delta_{jl}-K^+_{jk}\delta_{li},\qquad
 2\big[K^0_{ji},K^0_{kl}\big] = K^0_{jl}\delta_{ki}-K^0_{ki}\delta_{lj} . \label{baza24}
\end{gather}
\end{subequations}
Applying the same arguments as in the proof of Remark~3 in~\cite{csg} (see also the Appendix in~\cite{csg}, where the notion of homogeneous reductive space in the meaning of Nomizu~\cite{nomizu} is recalled; also see~\cite{SB14}) to the Jacobi algebra~$\got{g}^J_n$ determined by~\eqref{baza1M}, \eqref{baza3M}, \eqref{baza2M}, we make the

\begin{Remark}\label{geom}
The Jacobi group $G^J_n$ is an unimodular, non-reductive, algebraic group of Harish-Chandra type. The Siegel--Jacobi domain~$\mc{D}^J_n$ is a~reductive, non-symmetric manifold associated to the Jacobi group~$G^J_n$ by the generalized Harish-Chandra embedding.
\end{Remark}

We have attached to the Jacobi group $G^J_n:=\mr{H}_n\rtimes\operatorname{Sp}(n,\R)_{\C}$ coherent states based on the Siegel--Jacobi ball~$\mc{D}^J_n$~\cite{sbj}. Perelomov's CS vectors \cite{perG} associated to the group~$G^J_n$ with Lie algebra the Jacobi algebra $\got{g}^J_n$ based on the complex $d$-dimensional ($d= n(n+3)/2$) manifold~-- the Siegel--Jacobi ball~-- $ \mc{D}^J_n\approx\mr{H}_n/\R\times \operatorname{Sp}(n,\R )_{\C}/\text{U}(n)= \C^n\times\mc{D}_n $~-- have been def\/ined as \cite{sbj}
\begin{gather}\label{csuX}
e_{z,W}= \exp ({\mb{X}})e_0,
\qquad \mb{X} := \sqrt{\mu}\sum_{i=1}^n z_i \mb{a}^{\dagger}_i + \sum_{i,j =1}^n w_{ij}\mb{K}^+_{ij},\qquad
 z\in \C^n, \qquad W\in\mc{D}_n.
\end{gather}

The vector $e_0$ appearing in (\ref{csuX}) verif\/ies the relations
\begin{gather}
\mb{a}_ie_0= 0,\qquad i=1, \dots, n, \nonumber\\
\mb{K}^+_{ij} e_0 \not= 0 ,\qquad
\mb{K}^-_{ij} e_0 = 0 ,\qquad
\mb{K}^0_{ij} e_0 = \frac{k}{4}\delta_{ij} e_0, \qquad i,j=1,\dots,n .\label{vacuma}
\end{gather}

In (\ref{vacuma}) $e_0=e^H_0\otimes e^K_0$, where $e^H_0$ is the minimum weight vector (vacuum) for the Heisenberg group~$\mr{H}_n$, while $e^K_0$ is the extremal weight vector for $\operatorname{Sp}(n,\R)_{\C}$ corresponding to the weight~$k$ in~(\ref{vacuma}), and~$\mu$ parametrizes the Heisenberg group \cite{sbj,nou,gem}. Holomorphic irreducible representations of the Jacobi group based on Siegel--Jacobi domains have been studied in mathematics \cite{bb,bs,TA90,TA92,TA99}. In \cite{jac1,nou,gem} we have compared our results on the geometry of~$\mc{D}^J_n$ and holomorphic representations of~$G^J_n$ obtained using CS with those of the mentioned mathematicians.

The scalar product $K:\mc{D}^J_n\times \mc{D}^J_n \rightarrow\C$, $K(\bar{x},\bar{V};y,W)=(e_{x,V},e_{y,W})_{k\mu} $ is \cite{sbj,gem}:
\begin{gather}
(e_{x,V},e_{y,W})_{k\mu} = \det (U)^{k/2} \exp\mu
F(\bar{x},\bar{V};y,W), \qquad U= (\un-W\bar{V})^{-1}, \nonumber \\
 2 F(\bar{x},\bar{V};y,W) = 2(x, U y) +( V\bar{y},U y) +( x,U W\bar{x}) .\label{KHKX}
\end{gather}
If we denote $\zeta = (x,V)$, $\zeta' = (y,W)$, then we f\/ind for the normalized Bergman kernel \eqref{kmic} of the Siegel--Jacobi ball the expression
\begin{gather}\label{kapS}
\kappa_{\mc{D}^J_n}(\zeta,\zeta')=\kappa_{\mc{D}_n}(V,W)\exp\mu[
2F(\zeta,\zeta')- F(\zeta)-F(\zeta')],\end{gather}
where $\kappa_{\mc{D}_n}(V,W)$ is the normalized Bergman kernel of the Siegel ball
\begin{gather*}
\kappa_{\mc{D}_n}(V,W)= \det{}^{k/2}\left[\frac{(\un-V\bar{V})(\un-W\bar{W})}{(\un-W\bar{V})^2}\right]. \end{gather*}
We f\/ind for the Calabi diastasis \eqref{DIA} of the Siegel--Jacobi ball the expression
\begin{gather}\label{DiaS}
D_{\mc{D}^J_n}(\zeta,\zeta')=D_{\mc{D}_n}(V,W)+2\mu\big[\Re F(\zeta)+\Re F(\zeta')- 2\Re F(\zeta,\zeta')\big],\\
D_{\mc{D}_n}(V,W)=-2\ln |\kappa_{\mc{D}_n}(V,W)|.\nonumber
\end{gather}
In particular, the reproducing kernel $K(z,W)=(e_{z,W},e_{z,W})_{k\mu}$, $z\in\C^n$, $W\in\mc{D}_n$ is
\begin{gather}\label{kul}
K(z,W) = \det(M)^{k/2}\exp\mu F, \qquad M=(\un-W\bar{W})^{-1},
\\
2F =2\bar{z}^tMz+z^t\bar{W}Mz+\bar{z}^tMW\bar{z},\qquad 
2F = 2\bar{\eta}^t\eta -\eta^t\bar{W}\eta-\bar{\eta}^tW\bar{\eta},\nonumber\\
\label{etaZ}
\eta=M(z+W\bar{z}), \qquad z=\eta-W\bar{\eta}.
\end{gather}
The Hilbert space of holomorphic functions $\mc{F}_K$ associated to the kernel $K$ given by~\eqref{kul} is endowed with the scalar product of the type~\eqref{PAR}
\begin{gather*}
(\phi ,\psi ) _{\mc{F}_{K}}= \Lambda_n \int_{z\in\C^n;\,W=W^t;\,
1-W\bar{W}>0}\bar{f}_{\phi}(z,W)f_{\psi}(z,W)Q K^{-1} \dd z \dd W,\\
\dd z = \prod_{i=1}^n \dd \Re z_i \dd \Im z_i, \qquad \dd W = \prod_{1\le i\le j \le n} \dd \Re w_{ij} \dd \Im {w}_{ij},
\end{gather*}
 where the normalization constant $\Lambda_n$ is given by
\begin{gather*}
\Lambda_n = \mu^n\frac{k-3}{2\pi^{n(n+3)/2}}\prod_{i=1}^{n-1} \frac{\big(\frac{k-3}{2}-n+i\big)\Gamma (k+i-2)}{\Gamma [k+2(i-n-1)]} ,
\end{gather*}
and the density of volume is~\cite{sbj} (see also Lemma~\ref{lemmahua})
\begin{gather}\label{QQQ}
Q (z,W)= \det (\un-W\bar{W})^{-(n+2)}.
\end{gather}

The manifold $\mc{D}^J_n$ has the K\"ahler potential~\eqref{kelerX}, $f=\log K$, with $K$ given by~\eqref{kul},
\begin{gather}
f = -\tfrac{k}{2}\log \det (\un-W\bar{W})\nonumber\\
\hphantom{f=}{} +\mu \big\{ \bar{z}^t(\un-W\bar{W})^{-1}{z}
+ \tfrac{1}{2}z^t\big[\bar{W}(\un-W\bar{W})^{-1}\big]z +\tfrac{1}{2}\bar{z}^t \big[ (\un-W\bar{W})^{-1}W\big]\bar{z} \big\}.\label{kelerX}
\end{gather}
Because the symmetry of the matrix $W\in\mc{D}_n$, we write down the K\"ahler two-form on the Siegel--Jacobi ball as
\begin{gather}
-\ii \omega _{\mc{D}^J_n}(z,W) = h_{i\bar{j}}\dd z_i\wedge\dd \bar{z}_j+ \sum_{p\leq q}(h_{i\bar{p}\bar{q}}\dd z_i\wedge\dd
 \bar{w}_{pq}-\bar{h}_{i\bar{p}\bar{q}}\dd \bar{z}_i\wedge \dd w_{pq}) \nonumber\\
\hphantom{-\ii \omega _{\mc{D}^J_n}(z,W) =}{} + \sum_{p\le q; m\leq n}h_{pq\bar{m}\bar{n}}\dd
 w_{pq}\wedge \dd \bar{w}_{mn}.\label{omww}
\end{gather}
We use the following notation for the matrix of the metric
\begin{gather}\label{math}
 h= \left(\begin{matrix} h_1 &h_2\\h_3& h_4\end{matrix}\right) =
\left(\begin{matrix} h_{i\bar{j}}& h_{i\bar{p}\bar{q}} \\
 h_{pq\bar{i}} & h_{pq\bar{m}\bar{n}} \end{matrix}\right) \in
M(n(n+3)/2,\C),\\
 p\leq q, \qquad m\leq n, \qquad h=h^{*},\nonumber
\end{gather}
where $h_1\in M(n,\C)$, $h_2\in M(n\times n(n+1)/2,\C)$, $h_3\in
M(n(n+1)/2\times n,\C)$, $h_4\in M(n(n+1)/2,\C)$.

We use the following convention: indices of $z\in M(n\times 1,\C)$ are denoted with $i$, $j$, $k$, $l$; indices of $W=W^t$, $W\in M(n,\C)$ are denoted with $m$, $n$, $p$, $q$, $r$, $s$, $t$, $u$, $v$.

Let as write the ``inverse'' $h^{-1}$ of the matrix $h$ \eqref{math} as
\begin{gather}\label{matrh1}
h^{-1}=\left(\begin{matrix}h^1&h^2\\h^3&h^4\end{matrix}\right) =
\left(\begin{matrix}(h^1)_{i\bar{j}}&(h^2)_{i\bar{p}\bar{q}}\\(h^3)_{pq\bar{i}}&(h^4)_{pq\bar{m}\bar{n}}\end{matrix}\right), \qquad
p\leq q, \qquad m\leq n,
\end{gather}
where the matrices $h^1$, $h^2$, $h^3$, $h^4$ have the same dimensions as the matrices $h_1$, $h_2$, $h_3$, respectively~$h_4$.

In fact, we shall determine the ``inverse'' \eqref{matrh1} of \eqref{math} such that
\begin{gather}\label{ciudat}
\left(\begin{matrix} h_1 &h_2\\h_3& h_4\end{matrix}\right)
\left(\begin{matrix}h^1&h^2\\h^3&h^4\end{matrix}\right)=
\left(\begin{matrix} \un & 0\\ 0 & \Delta \end{matrix}\right),
\end{gather}
where $\Delta$ is def\/ined in \eqref{mircea}.

It is useful to introduce the notation
\begin{gather*}
 f_{pq}:=1-\tfrac{1}{2}\delta_{pq}.
\end{gather*}

We shall prove the following (partial results have been
presented in \cite[Proposition~11]{JGSP}, \cite[Proposition 3.11]{sbj} and \cite[Proposition 1]{nou}):
\begin{Theorem}\label{mainTH}
The K\"ahler two-form $\omega_{\mc{D}^J_n}$, associated with the balanced metric of the Siegel--Jacobi ball $\mc{D}^J_n$, $G^J_n$-invariant to the action~\eqref{TOIU} is
\begin{gather}
- \ii\omega_{\mc{D}^J_n}(z,W) = \tfrac{k}{2}\tr (\mc{B}\wedge\bar{\mc{B}})
 +\mu \tr (\mc{A}^t\bar{M}\wedge \bar{\mc{A}}), \qquad \mc{A} =\dd z+ \dd W\bar{\eta},\nonumber\\
\mc{B} = M\dd W,\qquad M = (\un-W\bar{W})^{-1}.\label{aabX}
\end{gather}
The matrix \eqref{math} of the metric on $\mc{D}^J_n$ has the matrix elements~\eqref{hcomp}:
\begin{subequations}\label{hcomp}
\begin{gather}
h_{i\bar{j}} =\mu \bar{M}_{ij},\label{hc11}\\
h_{i\bar{p}\bar{q}} = \mu (\eta_q\bar{M}_{ip}
+\eta_p\bar{M}_{iq})f_{pq},\label{hc12}\\
h_{pq\bar{i}} = \mu (\bar{\eta}_q\bar{M}_{pi} +\bar{\eta}_p\bar{M}_{qi})f_{pq},\label{hc21}\\
h_{pq\bar{m}\bar{n}} =\frac{k}{2}h^k_{pq\bar{m}\bar{n}}+\mu
h^{\mu}_{pq\bar{m}\bar{n}},\label{hc22}\\
h^k_{pq\bar{m}\bar{n}} =
2(M_{mp}M_{nq}+M_{mq}M_{np})-2M_{mp}(M_{np}\delta_{pq}+M_{mq}\delta_{mn})
 + M^2_{mp}\delta_{pq}\delta_{mn}\nonumber \\
\hphantom{h^k_{pq\bar{m}\bar{n}}}{}
 = 2M_{mp}M_{nq}(1-\delta_{pq})+2M_{mq}M_{np}(1-\delta_{mn})+ M^2_{mp}\delta_{pq}\delta_{mn} ,\label{hc23}\\
h^{\mu}_{pq\bar{m}\bar{n}} = \bar{\eta}_p(\eta_n M_{mq}+\eta_m
M_{nq})+\bar{\eta}_q({\eta}_nM_{mp}+\eta_mM_{np})
 -\bar{\eta}_p(\eta_n M_{mp}\nonumber\\
 \hphantom{h^{\mu}_{pq\bar{m}\bar{n}}=}{} +\eta_m
M_{np})\delta_{pq}-{\eta}_m(\bar{\eta}_pM_{mq}+\bar{\eta}_qM_{mp})\delta_{mn}
 +\bar{\eta}_p\eta_mM_{mp}\delta_{pq}\delta_{mn}\nonumber\\
 \hphantom{h^{\mu}_{pq\bar{m}\bar{n}}}{}
 = [\bar{\eta}_p(\eta_n\bar{M}_{qm}+\eta_m\bar{M}_{qn})+\bar{\eta}_q(\eta_n\bar{M}_{pm}+\eta_m\bar{M}_{pn})]f_{pq}f_{mn}.\label{hc24}
\end{gather}
\end{subequations}
The ``inverse'' $h^{-1}$ of the metric matrix $h$ which verif\/ies \eqref{ciudat}, with components \eqref{hcomp}, obtained with the inverse \eqref{kpqmn} of $h^k$ \eqref{hc23} has the elements $h^1$--$h^4$ given by
\begin{subequations}\label{nr11}
\begin{gather}
\big(h^1\big)_{ij}
 =\theta\bar{M}^{-1}_{ij},
 \qquad \theta = \frac{1}{\mu}+\alpha\frac{2}{k},\qquad
\alpha=\eta^t\bar{M}^{-1}\bar{\eta}=\bar{\eta}_nS_n, \qquad S_n=\sum\eta_q\bar{M}^{-1}_{qn},\!\!\label{NR1}\\
\big(h^2\big)_{i\bar{m}\bar{n}} =-\frac{1}{k}\big(S_n\bar{M}^{-1}_{im}+S_m\bar{M}^{-1}_{in}\big),\label{NR2}\\
\big(h^3\big)_{mn\bar{i}} =-\frac{1}{k}\big(\bar{S}_n\bar{M}^{-1}_{mi}+\bar{S}_m\bar{M}^{-1}_{ni}\big),\label{NR3}\\
\big(h^4\big)_{pq\bar{m}\bar{n}} =
\frac{2}{k}(h^k)^{-1}_{pq\bar{m}\bar{n}}=\frac{1}{k}\big(\bar{M}^{-1}_{qn}\bar{M}^{-1}_{pm}+\bar{M}^{-1}_{pn}\bar{M}^{-1}_{qm}\big). \label{NR4}
\end{gather}
\end{subequations}
The determinant of the metric matrix $h$ is
\begin{gather}\label{DTH}
\mc{G}_{\mc{D}^J_n}(z,W)= \det h _{\mc{D}^J_n}(z,W)= \left(\frac{k}{2}\right)^{\frac{n(n+1)}{2}}\mu^n\det
(\un-W\bar{W})^{-(n+2)}.\end{gather}
In the case $n=1$ formulas \eqref{hcomp}, \eqref{nr11}, \eqref{DTH} became the formulas \eqref{metrica}, \eqref{hinv}, respectively~\eqref{g311}, with $2k\leftrightarrow\frac{k}{2}$.
\end{Theorem}

\begin{proof}Firstly, we determine the matrix elements of the metric $h$ on the Siegel--Jacobi disk $\mc{D}^J_n$ applying formula~\eqref{Ter} to the K\"ahler potential~\eqref{kelerX} as in~\eqref{omww}.

We use the formulas (see \cite[equations~(4.23)]{nou})
\begin{subequations}\label{LIKK}
\begin{gather}
\frac{\pa M_{ab} }{\pa w_{ik}} = M_{ai}X_{kb}+ M_{ak}X_{ib}- M_{ai}X_{ib}\delta_{ik},\qquad \text{where} \quad X=X^t=\bar{W}M=\bar{M}\bar{W},\label{LIK}\\
\frac{\pa X_{ab} }{\pa w_{ik}} = X_{ai}X_{bk}+X_{ak}X_{ib}-X_{ai}X_{ib}\delta_{ik} , \\
\frac{\pa \bar{ X}_{ab} }{\pa w_{ik}} =M_{ai}M_{bk}+M_{ak}M_{bi}-M_{ai}M_{bk}\delta_{ik} \label{LIK2}.
\end{gather}
\end{subequations}

From \eqref{etaZ}, we get
\begin{subequations}\label{morD}
\begin{gather}
\frac{\pa \eta_q}{\pa z_l} = M_{ql},\\
\frac{\pa \bar{\eta}_q}{\pa z_j} ={X}_{qj}, \\
\frac{\pa \eta_t}{\pa w_{pq}} = M_{tp}\bar{\eta}_q +M_{tq}\bar{\eta}_p- M_{tp}\bar{\eta}_p\delta_{pq}, \\
\frac{\pa \bar{\eta}_n}{\pa w_{pq}} =
\bar{\eta}_pX_{qn}+\bar{\eta}_qX_{pn}-\bar{\eta}_qX_{pn}\delta_{pq} .
\end{gather}
\end{subequations}

With \eqref{mircea}, \eqref{LIKK} and \eqref{morD}, we f\/ind the matrix elements~\eqref{hcomp} of the metric of the Siegel--Jacobi ball.

Now we prove \eqref{hc23}. We calculate
\begin{gather}\label{HKJ} h^k _{mn\bar{p}\bar{q}}=\frac{\pa}{\pa
 \bar{w}_{pq}} J_{mn}, \qquad J_{mn}= \frac{\pa}{\pa w_{mn}}\ln\det M. \end{gather}
With formula
\begin{gather*}
\frac{\dd (\det A)}{\dd t}=\det A \tr\left(A^{-1}\frac{\pa A}{\pa t}\right),\end{gather*} we have
\begin{gather*}J_{mn}= M^{-1}_{ij}\frac{\pa M_{ji}}{\pa w_{mn}}.\end{gather*}
 We f\/ind \begin{gather*}J_{mn}=X_{nm}+X_{mn}-X_{mn}\delta_{mn}=2X_{mn}f_{mn}.\end{gather*}
In order to obtain $ h^k _{mn\bar{p}\bar{q}}$, we calculate
\begin{gather*} \frac{\pa X_{nm}}{\pa \bar{w}_{pq}}= \frac{\pa \bar{w}_{nk}}{\pa
 \bar{w}_{pq}}M_{kn}+w_{nk}\frac{\pa M_{kb}}{\pa \bar{w}_{pq}}.\end{gather*}
But with \begin{gather*}\frac{\pa M_{km}}{\pa
 \bar{w}_{pq}}=\overline{\left(\frac{M_{mk}}{\pa w_{pq}}\right)}, \qquad \text{and with} \qquad \un+\bar{W}\bar{X}=\bar{M},\end{gather*} we f\/ind
\begin{gather*} \frac{\pa X_{nm}}{\pa
 \bar{w}_{pq}}=2(M_{pm}\bar{M}_{nq}+M_{qm}\bar{M}_{np}-M_{pm}\bar{M}_{nq}\delta_{pq})
-(2M_{pm}\bar{M}_{mq}-M_{pm}\bar{M}_{mq}\delta_{pq})\delta_{mn},\end{gather*}
i.e., we reobtain formula \eqref{hc23}.

With \eqref{hc23}, we have
\begin{gather}
h^{k}_{pp\bar{m}\bar{m}} = M^2_{mp}, \nonumber\\
h^{k}_{pp\bar{m}\bar{n}} = 2M_{mp}M_{nq},\qquad m<n,\nonumber\\
h^{k}_{pq\bar{m}\bar{m}} = 2M_{mp}M_{mq},\qquad p<q,\nonumber\\
h^{k}_{pq\bar{m}\bar{n}} = 2(M_{mp}M_{nq}+M_{mq}M_{np}),\qquad p<q,\qquad m<n.\label{620Ec}
\end{gather}

In order to f\/ind $-\ii \frac{2}{k}\dd \omega_{\mc{D}_n}(W) $, we make
the summation
\begin{gather}
 \sum h^k_{pp\bar{m}\bar{m}}\dd w_{pp}\wedge\dd \bar{w}_{mm} + \sum_{m<n} h^k_{pp\bar{m}\bar{n}}\dd
w_{pp}\wedge\dd \bar{w}_{mn} \nonumber\\
\qquad\quad{} +\sum_{p<q} h^k_{pq\bar{m}\bar{m}}\dd
w_{pq}\wedge\dd \bar{w}_{mm} + \sum_{p<q,\, m<n} h^k_{pq\bar{m}\bar{n}}\dd
w_{pq}\wedge\dd \bar{w}_{mn}\nonumber\\
\qquad{} = \sum M_{mp}\bar{M}_{pm}\dd
w_{pp}\wedge\dd \bar{w}_{mm} + 2\sum_{m<n}M_{mp}M_{pn}\dd
w_{pp}\wedge\dd \bar{w}_{mn} \nonumber\\
\qquad\quad{} + 2\sum_{p<q}M_{mp}M_{mq} + 2\sum_{p<q, \, m<n}(M_{mp}M_{nq}+M_{mq}M_{np}) \dd
w_{pq}\wedge\dd \bar{w}_{mn} \nonumber\\
\qquad{} = \sum M_{mp}\bar{M}_{pm}\dd
w_{pp}\wedge\dd \bar{w}_{mm} + \sum_{m\not= n}M_{mp}M_{pn}\dd
w_{pp}\wedge\dd \bar{w}_{mn} \nonumber\\
\qquad\quad{} + \sum_{p\not= q}M_{mp}M_{nq}\dd w_{pq}\wedge \dd\bar{w}_{mm}+
\sum_{p\not= q,\, m\not= n}M_{mp}M_{nq} \dd
w_{pq}\wedge\dd \bar{w}_{mn}\nonumber \\
\qquad{} = (M\dd W)_{mq}\wedge(\bar{M}\dd\bar{W})_{qm}.\label{sumL}
\end{gather}
We write down \eqref{sumL} as
 \begin{gather}\label{omd}
-\ii \frac{2}{k}\dd \omega_{\mc{D}_n}(W) = \tr(M\dd W\wedge \bar{M}\dd\bar{W}).
\end{gather}
Now we make the summation
\begin{gather*}
\sum_{p\le q}h^{\mu}_{i\bar{p}\bar{q}}\dd z_i\wedge
 \dd\bar{w}_{pq} = \sum h^{\mu}_{i\bar{p}\bar{p}}\dd z_i\wedge
 \dd\bar{w}_{pp} + \sum_{p< q}h^{\mu}_{i\bar{p}\bar{q}}\dd z_i\wedge
 \dd\bar{w}_{pq} \\
 \hphantom{\sum_{p\le q}h^{\mu}_{i\bar{p}\bar{q}}\dd z_i\wedge \dd\bar{w}_{pq}}{}
 = \bar{M}_{ip}\eta_p\dd z_i\wedge \dd
\bar{w}_{pp}+\sum_{p<q}(\bar{M}_{ip}\eta_q+\bar{M}_{iq}\eta_p)\dd
z_i \wedge \dd\bar{w}_{pq} \\
 \hphantom{\sum_{p\le q}h^{\mu}_{i\bar{p}\bar{q}}\dd z_i\wedge \dd\bar{w}_{pq}}{} = (M\dd z)_p\wedge(\dd \bar{W}\eta)_p.
\end{gather*}
We calculate the sum \begin{gather*}H^{\mu}:=\sum_{p\le q,m\le n}h^{\mu}_{pq\bar{m}\bar{n}}\dd
w_{pq}\wedge \dd \bar{w}_{mn}. \end{gather*}
We have
\begin{gather*}
H^{\mu} = \sum h^{\mu}_{pp\bar{m}\bar{m}}\dd
w_{pp}\wedge \bar{w}_{mm} + \sum_{m<n}h^{\mu}_{pp\bar{m}\bar{n}}\dd
w_{pp}\wedge \bar{w}_{mn} \\
\hphantom{H^{\mu} =}{} + \sum_{p<q}h^{\mu}_{pq\bar{m}\bar{m}}\dd
w_{pq}\wedge \bar{w}_{mm} + \sum_{p< q,\, m< n}h^{\mu}_{pq\bar{m}\bar{n}}\dd
w_{pq}\wedge \bar{w}_{mn} ,
\end{gather*}
and f\/inally we get
\begin{gather}\label{omdLL}
H^{\mu} = M_{mq}\dd w_{qp}\wedge \dd \bar{w}_{mn}\eta_n= (M\dd
W\bar{\eta})_m\wedge (\dd \bar{W}\eta)_m.
\end{gather}
Putting together \eqref{omd}, \eqref{sumL} and \eqref{omdLL}, we have proved~\eqref{aabX}.

In order to check out the homogeneity of the K\"ahler two-form $\omega_{\mc{D}^J_n}$ \eqref{aabX}, we use the formulas
 \begin{gather}\label{DW1}
\dd W_1=(Wq^*+p^*)^{-1}\dd W (\bar{q}W+\bar{p})^{-1},\\
\label{DM1} M_1=(q\bar{W}+p)M(W q^*+p^*),\\
\dd z_1= (Wq^*+p^*)^{-1}(\dd W V +\dd z), \qquad V=-q^*(Wq^*+p^*)^{-1}[\eta+\alpha-W(\bar{\eta}+\bar{\alpha})]-\bar{\alpha}, \nonumber\\
 \eta_1=p(\alpha+\eta)+q(\bar{\alpha}+\bar{\eta}),\qquad
\mc{A}_1= (Wq^*+p^*)^{-1}\mc{A}.\nonumber\end{gather}

Now we calculate the ``inverse'' \eqref{matrh1} $h^{-1}$ of the matrix $h$ whose elements are given by \eqref{hcomp}.

We have (see, e.g., in \cite[(1) and (2), p.~30]{lutke}): if the matrices $h_1$, $h_4-h_3h^{-1}_1h_2$, $h_4$ and respectively $h_1-h_2h^{-1}_4h_3$
are nonsingular, then
\begin{subequations}\label{inv}
\begin{gather}
h^1 = \big(h_1-h_2h^{-1}_4h_3\big)^{-1} 
 = h_1^{-1}+ h_1^{-1}h_2h^4h_3 h_1^{-1},\label{inv11}\\
h^2 = -h^{-1}_1h_2h^4,\label{inv2}\\
h^3 = - h^4 h_3h_1^{-1},\label{inv31}\\
h^4 = (h_4-h_3h^{-1}_1h_2)^{-1}.\label{inv4}
\end{gather}
\end{subequations}
In order to calculate $h^4$, we write~\eqref{inv4} as \begin{gather}\label{NR5}(h^4)^{-1}=h_4-\mu h^I, \qquad\text{where} \quad
h^I:=\frac{1}{\mu}h_3h_1^{-1}h_2.\end{gather}
With formulas \eqref{hc11}--\eqref{hc21}, we get
\begin{gather*}
\big(h^I\big)_{pq\bar{m}\bar{n}} =h_{pq\bar{i}}(h_1)^{-1}_{ij}h_{j\bar{m}\bar{n}} =
(\bar{\eta}_q\bar{M}_{pi}+\bar{\eta}_p\bar{M}_{qi})\bar{M}^{-1}_{ij}(\eta_n\bar{M}_{jm}+\eta_m\bar{M}_{jn})
f_{pq}f_{mn}\\
\hphantom{\big(h^I\big)_{pq\bar{m}\bar{n}}}{}
=(\bar{\eta}_q\bar{M}_{pi}+\bar{\eta}_p\bar{M}_{qi})(\eta_n\delta_{im}+\eta_m\delta_{im})f_{pq}f_{mn}\\
\hphantom{\big(h^I\big)_{pq\bar{m}\bar{n}}}{}
 = [\bar{\eta}_q(\eta_n\bar{M}_{pm}+\eta_m\bar{M}_{pn})+\bar{\eta}_p(\eta_n\bar{M}_{qm}+\eta_m\bar{M}_{qn})]f_{pq}f_{mn},
\end{gather*}
and with equation \eqref{hc24} we get
\begin{gather*}h^I=h^{\mu}.\end{gather*}
With \eqref{NR5}, we obtain
\begin{gather}\label{h4h1}
\big(h^4\big)^{-1}=\frac{k}{2}h^k,\end{gather}
and \eqref{NR4}, where $h^k$ has the expression~\eqref{hc23}, while $(h^k)^{-1}$ has the expression \eqref{kpqmn}.

In order to calculate $h^2$, taking into account \eqref{h4h1}, we write~\eqref{inv2} as
\begin{gather}\label{hjj}
h^2=-(h_1)^{-1}h^J,\qquad\text{where} \quad h^J:=\frac{2}{k}h_2\big(h^k\big)^{-1}. \end{gather}
With formulas \eqref{hc12}, \eqref{ofi}, we have
\begin{gather*}
\big(h^J\big)_{j\bar{m}\bar{n}} =2\frac{\mu}{k}\sum_{p\leq
 q}f_{pq}(\eta_q\bar{M}_{jp}+\eta_p\bar{M}_{jq})k_{pq\bar{m}\bar{n}}\\
\hphantom{\big(h^J\big)_{j\bar{m}\bar{n}}}{}
= 2\frac{\mu}{k}\left[\sum_{p}\eta_p(\un-W\bar{W})^{-1}_{pj}(\un-W\bar{W})_{np}(\un-\bar{W}W)_{pm}\right]\\
\hphantom{\big(h^J\big)_{j\bar{m}\bar{n}}=}{}+
\frac{\mu}{k}\sum_{p<q}\big[\eta_q(\un-W\bar{W})^{-1}_{pj}+\eta_p(\un-W\bar{W})^{-1}_{pj}\big] \\
\hphantom{\big(h^J\big)_{j\bar{m}\bar{n}}=}{}
\times\big[(\un-W\bar{W})_{nq}(\un-\bar{W}W)_{pm}+(\un-W\bar{w})_{np}(\un-\bar{W}W)_{qm}\big]\\
\hphantom{\big(h^J\big)_{j\bar{m}\bar{n}}}{}
= \frac{\mu}{k}\left[\sum_{p\not=q}(\eta_q(\un-W\bar{W})^{-1}_{pj}(\un-W\bar{W})_{nq}(\un-\bar{W}W)_{pm}\right.\\
\hphantom{\big(h^J\big)_{j\bar{m}\bar{n}}=}{}
+\sum_{p\not=q}\eta_p(\un-W\bar{W})^{-1}_{qj}(\un-W\bar{W})_{nq}(\un-\bar{W}{W})_{pm}\\
\left.\hphantom{\big(h^J\big)_{j\bar{m}\bar{n}}=}{}
+2\sum_p\eta_p(\un-W\bar{W})^{-1}_{pj}(\un-W\bar{W})_{np}(\un-\bar{W}W)_{pm}\right]\\
\hphantom{\big(h^J\big)_{j\bar{m}\bar{n}}=}{}
=\frac{\mu}{k}(S_n\delta_{jm}+S_m\delta_{jn}),\qquad \text{where} \quad S_n:=\sum_q\eta_q\bar{M}^{-1}_{qn}.
\end{gather*}

With \eqref{hjj}, we get \eqref{NR2} and similarly for \eqref{NR3}.

In order to calculate $h^1$, we write \eqref{inv11} as
\begin{gather*}
h^1=h^{-1}_1-h^Kh^{-1}_1,\qquad\text{where} \quad h^K:= h^2h_3.\end{gather*}
With formulas \eqref{NR2} and \eqref{hc21}, we get
\begin{gather*}
\big(h^K\big)_{ik} =-\frac{\mu}{k}\sum_{m\leq
 n}f_{mn}\big(S_n\bar{M}^{-1}_{im}+S_m\bar{M}^{-1}_{in}\big)(\bar{\eta}_n\bar{M}_{mk}+\bar{\eta}_m\bar{M}_{nk})\\
\hphantom{\big(h^K\big)_{ik} }{} =
-\frac{\mu}{k}\left[\frac{1}{2}\sum_m4S_m\bar{\eta}_m\bar{M}^{-1}_{im}\bar{M}_{mk}
 + \sum_{m< n}\big(S_n\bar{M}^{-1}_{im} + S_m\bar{M}^{-1}_{in}\big)\big(\bar{\eta}_n\bar{M}_{mk} + \bar{\eta}_m\bar{M}_{nk}\big)\right]\\
\hphantom{\big(h^K\big)_{ik} }{}= -\frac{2}{k}\mu\alpha\delta_{ik},
\end{gather*}
where \begin{gather*}\alpha:=\eta^t\bar{M}^{-1}\bar{\eta}=\eta^*M^{-1}\eta=\bar{\eta}_nS_n.\end{gather*}
With \eqref{hc11}, we get \eqref{NR1}.

In order to prove \eqref{DTH}, we use the relation
\begin{gather*}\det h=\det \left(\begin{matrix} h_1 &h_2\\h_3&
 h_4\end{matrix}\right)
 =\det h_1 \det \big(h_4-
 h_3h^{-1}_1h_2\big).\end{gather*}
With \eqref{inv4} and \eqref{h4h1}, we get
\begin{gather*}\det h =\left(\frac{k}{2}\right)^n\det h_1\det h^k.\end{gather*}
\eqref{DTH} is obtained introducing in the above relation the expressions~\eqref{hc11} and~\eqref{Qq}.
\end{proof}

Now we formulate several geometric properties of the Siegel--Jacobi ball relevant for the Berezin quantization of this manifold:
\begin{Proposition}\label{geoPL}\quad
\begin{enumerate}\itemsep=0pt
 \item[$i)$] The homogeneous K\"ahler manifold $\mc{D}^J_n$ is contractible.
\item[$ii)$] The K\"ahler potential of the Siegel--Jacobi ball~\eqref{kelerX} is global. $\mc{D}^J_n$ is a Q.-K.~Lu manifold, with normalized Bergman kernel \eqref{kapS} and nowhere vanishing diastasis~\eqref{DiaS}.
\item[$iii)$] The manifold $\mc{D}^J_n$ is a quantizable manifold.
\item[$iv)$] $\mc{D}^J_n$ is projectively induced and the Jacobi group $G^J_n$ is a CS-type group.
\item[$v)$] $\mc{D}^J_n$ is a homogeneous reductive space.
\end{enumerate}
\end{Proposition}
\begin{proof}
$i)$ We have proved in Theorem \ref{mainTH} (see also~\cite{JGSP,nou}) that $\mc{D}^J_n$ is a homogeneous K\"ahler manifold. We have shown in \cite[Proposition~4.1]{nou} that the ${\rm FC}$-transform ${\rm FC}\colon (\eta,W)\rightarrow (z,W)$ expressed by~\eqref{etaZ} is a~homogeneous K\"ahler dif\/feomorphism $\mc{D}^J_n\approx\mc{D}_n\times\C^n$, where the Siegel ball admits the realization~\eqref{dn}. But $\mc{D}_n$ can be achieved as the open ball $\mc{B}_n$ (see, e.g., \cite[p.~502]{neeb})
\begin{gather*}\mc{B}_n=\big\{W\in M(n,\C),\qquad W=W^t\, |\, ||W||<1\big\}.\end{gather*}
So, $\mc{D}^J_n$ is dif\/feomorphic with the product of the contractible spaces $\mc{D}_n$ and $\C^n$, and consequently, it is itself contractible.

$ii)$ We have for \eqref{kul} $K(z,W)>0$, $\forall\, (z,W)\in \mc{D}^J_n$. The explicit expressions of the K\"ahler potential, normalized
Bergman kernel and diastasis imply the assertions of $ii)$, but they could be also derived from the Theorem~\ref{LMM}, once we have proved~$i)$. Even more, Theorem \ref{LMM} asserts that $\mc{D}^J_n$ admits a~Berezin quantization.

$iii)$, $iv)$ We have observed in Theorem~\ref{mainTH} that the K\"ahler two-form \eqref{aabX} is associated with the balanced metric on $\mc{D}^J_n$, and we apply Remark~\ref{kra}.

$v)$ This assertion was already mentioned in Remark~\ref{geom}, where it was used the explicit form~\eqref{baza1M},~\eqref{baza3M}, \eqref{baza2M} of the Jacobi algebra $\got{g}^J_n$. We mention it again here because in~\cite{last} (see Remark~1 there for a more precise formulation) we have proved that {\it the
 CS-manifolds are reductive spaces}.
\end{proof}

The fact that the Jacobi algebra $\got{g}^J_n$ is a CS-algebra is well known, see, e.g., \cite[Theorem~5.2]{lis} or \cite[Example~VII.2.3 , p.~294]{neeb}. Using the explicit form of the base~$\Phi(z,w)$ of orthonormal polynomials in which the Bergman kernel $K_{\mc{D}^J_1}(z,w)$ is developed~\cite{jac1}, we have proved in \cite[Proposition~2]{berr} that the Siegel--Jacobi disk is a CS-manifold. In~\cite{gem} we have determined the base $\Phi(z,W)$ on $\mc{D}^J_n$ in which the reproducing kernel~\eqref{KHKX} admits a series expansion as in~\eqref{funck}, but in Proposition~\ref{geoPL} we have used Theorems~\ref{LMN},~\ref{LMM} and Remark~\ref{kra} to prove directly Proposition~\ref{geoPL}.

\subsection{Ricci form and scalar curvature}\label{SCR}
The Ricci form associated to the K\"ahlerian two-form $\omega_M$ \eqref{kall} is (see
\cite[p.~90]{mor})
\begin{gather}\label{RICCI}
\rho_M(z):=\ii \sum_{\alpha,\beta=1}^n\text{Ric}_{\alpha\bar{\beta}}(z)\dd
z_{\alpha}\wedge\dd \bar{z}_{\beta}, \qquad \text{Ric}_{\alpha\bar{\beta}}(z)=
 -\frac{\pa^2}{\pa z_{\alpha}\pa \bar{z} _{\beta}} \ln\mc{G}_M(z).
\end{gather}
The scalar curvature at a point $p\in M$ of coordinates $z$ is
(see \cite[p.~294]{koba1} or \cite[p.~144]{jost})
\begin{gather}\label{scc}
s_M(p):=\sum_{\alpha,\beta=1}^n (h_{\bar{\alpha}{\beta}})^{-1}\text{Ric}_{\alpha\bar{\beta}}(z).
\end{gather}
The Bergman metric corresponds to the K\"ahler two-form (see references in \cite{berr})
\begin{gather}\label{bergom1}
\omega^1_M=\ii \pa \bar{\pa} \ln \mc{G}.
\end{gather}

Q.-K.~Lu~\cite{LU08} has introduced for a bounded domain the positive def\/inite (1,1)-form
\begin{gather}\label{RIC2}
\tilde{\omega}_M(z):=
\ii\sum_{\alpha,\beta=1} ^n\tilde{h}_{\alpha\bar{\beta}} (z) \dd z_{\alpha}\wedge
\dd\bar{z}_{\beta}, \qquad \tilde{h}_{\alpha\bar{\beta}} (z):=
 (n+1)h_{\alpha\bar{\beta}} (z)- \text{Ric}_{\alpha,\beta}(z),
\end{gather}
which is K\"ahler, corresponding to the K\"ahler potential $\tilde{f} = \ln (K _M (z)^{n+1}\mc{G}(z))$.

Now we prove
\begin{Proposition}\label{mainPR}\quad
\begin{enumerate}\itemsep=0pt
\item[$i)$] The Siegel--Jacobi ball $\mc{D}^J_n$ is not an Einstein manifold with respect to the homogeneous K\"ahler metric attached to the K\"ahler two-form~\eqref{aabX}, but it is one with respect to the Bergman metric corresponding to the Bergman K\"ahler two-form~\eqref{bergom1}.

\item[$ii)$] The Ricci form \eqref{RICCI} on the Siegel--Jacobi ball associated to the homogeneous K\"ahlerian two-form \eqref{aabX} has the expression
\begin{gather*}
\rho_{\mc{D}^J_n}(z,W)=-\omega^1_{\mc{D}^J_n}(z,W) =-\ii (n+2)\tr (M\dd W \wedge \bar{M}\dd \bar{W}),
\end{gather*}
while
\begin{gather*}
\rho_{\mc{D}_n}(W)=-\omega^1_{\mc{D}_n}(W)=-\ii (n+1)\tr (M\dd W \wedge \bar{M}\dd
 \bar{W}).\end{gather*}

\item[$iii)$] The scalar curvature \eqref{scc} is constant and negative
\begin{gather*}
s_{\mc{D}^J_n}(z,W)= - \frac{2}{k}n\frac{(n+1)(n+2)}{2}.
\end{gather*}
\item[$iv)$] The Q.-K.~Lu K\"ahler two-form \eqref{RIC2}
has the expression
\begin{gather*}\tilde{\omega}_{\mc{D}^J_n}(z,W)=\frac{(n+1)(n+2)}{2}\omega_{\mc{D}^J_n}(z,W)-\rho_{\mc{D}^J_n}(W).\end{gather*}
\end{enumerate}
\end{Proposition}

\begin{proof}
With formula \eqref{DTH} (or \eqref{QQQ}) and \eqref{HKJ}, we f\/ind
that the only nonzero components of the Ricci tensor are
\begin{gather}\label{RiRi}
(\text{Ric}_{mn\bar{p}\bar{q}})_{\mc{D}^J_n}(z,W)= -(n+2)h^k_{mn\bar{p}\bar{q}}(W).
\end{gather}
With calculation \eqref{sumL}, we f\/ind
\begin{gather*}\rho_{\mc{D}^J_n}(z,W)=-\ii (n+2)\bar{M}_{pi}\bar{M}_{jq}\dd w_{pj}\wedge
 \dd \bar{w}_{qi}.\end{gather*}
Applying \eqref{scc} for $\mc{D}^J_n$, with formulas \eqref{NR4}, \eqref{RiRi}, we get
\begin{gather*}
s_{\mc{D}^J_n}(z,W) :=\sum_{\alpha,\beta =1}^n
 (h)^{-1}_{\bar{\alpha}{\beta}}\text{Ric}_{\alpha\bar{\beta}}(z,W)=\sum_{p\leq
 q;\, m\leq n}(h)^{-1}_{mn\bar{p}\bar{q}} (\text{Ric}_{pq\bar{m}\bar{n}}) _{\mc{D}^J_n}\\
\hphantom{s_{\mc{D}^J_n}(z,W)}{}
 = -\frac{2}{k}(n+2) \sum_{p\le q;\, m\leq n}(h^k)^{-1}_{mn\bar{p}\bar{q}}
(h^k)_{pq\bar{m}\bar{n}} = -\frac{2}{k}(n+2) \sum_{m\le n}\Delta^{mn}_{mn}\\
\hphantom{s_{\mc{D}^J_n}(z,W)}{}
 = -\frac{2}{k}(n+2)\frac{n(n+1)}{2}.\tag*{\qed}
\end{gather*}
\renewcommand{\qed}{}
\end{proof}

The scalar curvature of the Siegel--Jacobi disk was calculated in~\cite{csg,berr, jae}.

\subsection{Laplace--Beltrami operator}\label{LBO}
We introduce the formulas \eqref{nr11} into the expression \eqref{LPB} of the Laplace--Beltrami operator and we use Remark~\ref{equivv}. In the formula \eqref{LABDL} of the Laplace--Beltrami operator on $\mc{D}^J_n$, we use the expression given in Proposition~\ref{LLB} for the Laplace--Beltrami operator on $\mc{D}_n$. We introduce the
notation (no summation!)
\begin{gather}\label{symbe}
(D_z)_{\mu \nu}=\left(e_{\mu \nu}\frac{\pa}{\pa z_{\mu\nu}}\right), \qquad e_{\mu\nu}= \frac{1+\delta_{\mu\nu}}{2},~z_{\mu\nu}=z_{\nu\mu}.
\end{gather}

We obtain:
\begin{Proposition}\label{incaun}The Laplace--Beltrami operator on the Siegel--Jacobi ball $\mc{D}^J_n$, invariant to the action
 \eqref{TOIU} of the
 Jacobi group $G^J_n$, has the expression
\begin{gather}\label{LABDL}\Delta_{\mc{D}^J_n}(z,W) = \theta\frac{\pa}{\pa
 z^t}N\frac{\pa}{\pa\bar{z}}-\frac{2}{k}\left\{\tr \left[S D_W
 N\frac{\pa}{\pa\bar{z}}\right]+cc\right\}+\frac{2}{k}\Delta_{\mc{D}_n}(W),\\
 N=\un-W\bar{W}.\nonumber
\end{gather}

We have also the relations
\begin{gather}\label{lastB}
\Delta_{\mc{D}^J_n}(\ln \mc{G}) = -s_{\mc{D}^J_n} = \frac{2}{k}n\frac{(n+1)(n+2)}{2}.
\end{gather}
\end{Proposition}
\begin{proof}
We apply formula \eqref{LPB} with the matrix elements~\eqref{nr11} and we use Remark~\ref{equivv}. 
\end{proof}

We recall that Theorem~2.5 in Berezin's paper~\cite{ber74} asserts essentially that
\begin{gather*}\Delta_M(z)(\ln(\mc{G}(z)))= ct \end{gather*}
for the balanced metric plus other~3 conditions. The f\/irst equation in~\eqref{lastB}
\begin{gather*}s=-\Delta\ln \det h \end{gather*}
 is a general recipe for calculation of scalar curvature, see, e.g., \cite[equation~(5.2.24), p.~253]{jost}.

Also we recall that the Laplace operator on the Siegel--Jacobi ball was calculated in~\cite{Y10} in other coordinates. We have determined the Laplace--Beltrami operator on the Siegel--Jacobi disk in~\cite{berr}.

\appendix
\section{Appendix}\label{APP1}

\subsection{A remark}
\begin{Remark}\label{equivv}
Let $\omega_M(z)$ be the non-degenerate two-form \eqref{kall} on a~complex manifold $M$, invariant to a invertible holomorphic transformation $z'=z'(z)$. Then the dif\/ferential opera\-tor~$\Delta_M$ given by~\eqref{LPB} is also invariant to the transformation $z\rightarrow z'$. In particular, $M$~may be a~homogeneous K\"ahler manifold and $\Delta_M$ the corresponding Laplace--Beltrami operator. \end{Remark}

\begin{proof}
Let us consider the action $G \times M\rightarrow M$: $g\times z = z'$.
Then we have \begin{gather*}-\ii \omega_M(z')=h'_{\lambda\bar{\gamma}}(z')\dd
z'_{\lambda}\wedge \dd \bar{z}_{\gamma}
=h'_{\lambda\bar{\gamma}}(z'(z))a_{\lambda\alpha}\bar{a}_{\gamma\beta}\dd
z_{\alpha}\wedge \dd\bar{z}_{\beta},\qquad \text{where} \quad a_{\alpha\lambda}= \frac{\pa z'_{\alpha}}{\pa z_{\lambda}}.\end{gather*}
If $\omega_M(z')=\omega_M(z)$, then
\begin{gather}\label{cnd1}
h_{\alpha\bar{\beta}}(z)=h'_{\lambda\gamma}(z'(z))a_{\lambda\alpha}\bar{a}_{\gamma\beta}.
\end{gather}
Now we consider the Laplace--Beltrami operator \eqref{LPB} in the point
$z'$. We have \begin{gather*}\frac{\pa}{\pa
 z'_{\alpha}}=b_{\gamma\alpha}\frac{\pa}{\pa z_{\gamma}},\qquad \text{where} \quad b_{\lambda\beta}=\frac{\pa z_{\lambda}}{\pa z'_{\beta}},\end{gather*}
and \begin{gather*}a_{\alpha\lambda}b_{\lambda\beta}=\delta_{\alpha\beta}.\end{gather*}
We have
\begin{gather*}\Delta_M(z')=(h' (z'))^{-1}_{\alpha\bar{\beta}}\bar{b}_{\gamma\alpha}b_{\rho\beta}\frac{\pa^2}{\pa
\bar{z}_{\gamma}\pa z_{\rho}}.\end{gather*}
If $\Delta_M(z')=\Delta_M(z)$, then
\begin{gather}\label{cnd2}
h^{-1}_{\gamma\bar{\rho}}(z)=(h' (z'(z)))^{-1}_{\alpha\bar{\beta}}\bar{b}_{\gamma\alpha}b_{\rho\beta}.
\end{gather}
We write \eqref{cnd1} as
\begin{gather*}h=a^th'\bar{a},\end{gather*}
which implies
\begin{gather*}h^{-1}=(\bar{a})^{-1}(h')^{-1}(a^t)^{-1}=\bar{b}(h')^{-1}b^t,\end{gather*}
which is exactly~\eqref{cnd2}.
\end{proof}

\subsection{A lemma}
Formula \eqref{DTH}, modulo the numerical factor, was obtained in~\cite{sbj}. We recall the method to determine the Liouville form
used in \cite{sbj}. We have applied the following technique (see \cite[Chapter~IV]{hua}):
\begin{Lemma}\label{lemmahua}
Let $z'=f (g,z)$ denote the action of the group $G$ on the circular domain $M$. Let us determine the element $g\in G$ such that $z'(z_1)=0$ Then the density of the volume form is $Q= \vert J\vert^2$, where $J$ is the Jacobian $J=\frac{\pa z'}{\pa z}$.
\end{Lemma}

The transformation with the desired properties is
\begin{gather*}
z'= U(1-W_1\bar{W}_1)^{1/2}(1-W\bar{W}_1)^{-1}\\
\hphantom{z'= }{} \times \big[z-(1-W\bar{W}_1)(1-W_1\bar{W_1})^{-1}z_1 +
(W-W_1)(1-\bar{W}_1W_1)^{-1}\bar{z}_1\big], \\
 W' = U (1-W_1\bar{W}_1)^{-1/2}(W-W_1) (1-\bar{W}_1W)^{-1}(1-\bar{W}_1W_1)^{1/2}U^t, 
\end{gather*}
where $U$ is an unitary matrix. We f\/ind that
\begin{gather}
\frac{\pa z'}{\pa z}= U(1-W_1\bar{W}_1)^{1/2}(1-W\bar{W}_1)^{-1},\nonumber\\
\label{tt2}
\dd W'=A \dd W A^t, \qquad A = U(1-W_1\bar{W}_1)^{1/2}(1-W\bar{W}_1)^{-1}.
\end{gather}

In order to calculate the Jacobian of the transformation \eqref{tt2}, we use the following property extracted from Berezin's paper \cite[p.~398]{ber75}:

{\em Let $A$ be a matrix and $L_A$ the transformation of a matrix of the same order $n$, $L_A\xi = A\xi A^t$. If the matrices $A$ and $\xi$ are symmetric, then $\det L_A= (\det A)^{n+1}$}. We f\/ind as in~\cite{sbj}, in accord with formula before Theorem~4.3.2
in~\cite{hua} \begin{gather}\label{Qq}
Q_{\mc{D}_n}=\det(\un-W\bar{W})^{-(n+1)},
\end{gather}
while $Q_{\mc{D}^J_n}(z,W)$ has the expression~\eqref{QQQ}.

\section[Appendix Laplace--Beltrami operator on the Siegel upper half-plan and Siegel ball]{Appendix\\ Laplace--Beltrami operator on the Siegel upper half-plan\\ and Siegel ball} \label{APP2}

The Laplace--Beltrami operator on $\mc{X}_n$ has the expression (see \cite[equation~(19)]{maa})
\begin{gather}\label{oarecum}\Delta_{\mc{X}_n}=-\tr (Z-W)((Z-W)D_w)^tD_z,
\qquad Z=X+\ii Y \in \mc{X}_n, \qquad W=\bar{Z},
\end{gather}
where we use the notation \eqref{symbe}. Formula (43) in~\cite{maa} reads
\begin{gather}\label{DDW}
D_w(Z-W)=-\frac{n+1}{2}\un +((Z-W)D_w)^t,\end{gather}
where
\begin{gather}
D_w(Z-W) =\sum_{\rho=1}^n\e_{\mu\rho}\frac{\pa}{\pa w_{\mu\rho}}(z_{\rho\nu}-w_{\rho\nu})\nonumber\\
\hphantom{D_w(Z-W) }{} =
\left(\sum_{\rho=1}^n(z_{\mu\rho}-w_{\mu\rho})e_{\rho\nu}\frac{\pa}{\pa w_{\rho\nu}}\right)^t-\left(\delta_{\mu\nu}\sum_{\rho=1}^ne_{\mu\rho}\right).
\label{ma1}
\end{gather}

With formula \eqref{mircea}, introducing in \eqref{ma1}, it is obtained
\begin{gather*}
((Z-W)D_w)^t_{\mu\nu}(f)=\left[\frac{\pa f}{\pa W}(Z-W)\right]_{\mu\nu}=(z_{\rho \nu}-w_{\rho\nu})e_{\mu\rho}\frac{\pa
f}{\pa w_{\mu\rho}}. \end{gather*}
Equation (54) of the Cayley transform~\cite{maa} in the notation \eqref{DDW} reads
\begin{gather*}\tilde{Z}=(Z-\ii\un)(Z+\ii \un)^{-1}.\end{gather*}
 Equation~(55) in \cite{maa} asserts that
\begin{gather*}
D_z =-\frac{\ii}{2}(\tilde{Z}-\un)((\tilde{Z}-\un)D_{\tilde{z}})^t,\end{gather*}
i.e., if now $W$ describes a point \eqref{dn} in $\mc{D}_n$, we have the formula
\begin{Lemma}\label{masH}
\begin{gather}
(D_z)_{\alpha\beta} :=e_{\alpha\beta}\frac{\pa}{\pa
 z_{\alpha\beta}} =-\frac{\ii}{2}\big[(\un-W)((\un-W)D_w)^t\big]_{\alpha\beta} \nonumber\\
\hphantom{(D_z)_{\alpha\beta}}{} =
-\frac{\ii}{2}(\un-W)_{\alpha\gamma}(\un-W)_{\beta\lambda}e_{\lambda\gamma}\frac{\pa}{\pa
 w_{\lambda\gamma}}.\label{DZZ1}
\end{gather}
\end{Lemma}
\begin{proof}
We write \eqref{scg1} as
\begin{gather*} 2\ii \dd W= A \dd Z A, \qquad A=\un-W, \qquad W=W^t, \qquad Z=Z^t,\end{gather*}
i.e., \begin{gather*}2\ii \dd w_{pq}=A_{pm}\dd z_{mn}A_{nq}.\end{gather*}
With formula \eqref{mircea}, we get
\begin{gather}
2\ii \frac{\pa w_{pq}}{\pa z_{\alpha\beta}} = A_{pm}A_{nq}\frac{\pa z_{mn}}{\pa
 z_{\alpha\beta}}= A_{pm}A_{nq} (\delta_{m\alpha}\delta_{n\beta}+\delta_{m\beta}\delta_{n\alpha}-\delta_{mn}\delta_{\alpha\beta}\delta_{m\alpha})\nonumber\\
\hphantom{2\ii \frac{\pa w_{pq}}{\pa z_{\alpha\beta}}}{}
= A_{p\alpha}A_{\beta q}+A_{p\beta}A_{\alpha q}-A_{p\alpha}A_{\beta q}\delta_{\alpha\beta}.\label{ref56}
\end{gather}
We have the formula
\begin{gather*}\frac{\pa f}{\pa z_{\alpha\beta}}=\sum_{p\le q}\frac{\pa f}{\pa
 w_{pq}}\frac{\pa w_{pq}}{\pa z_{\alpha\beta}}.\end{gather*}
For $\alpha\not=\beta$, we get from \eqref{ref56}:
\begin{gather}
2\ii \frac{\pa f}{\pa z_{\alpha\beta}} =
\sum_{p<q}(A_{p\alpha}A_{\beta q}+A_{p\beta}A_{\alpha q})\frac{\pa
 f}{\pa w_{pq}}+\sum_{p=q}(A_{p\alpha}A_{\beta p}+A_{p\beta}A_{\alpha
 p})\frac{\pa f}{\pa w_{pq}}\nonumber\\
\hphantom{2\ii \frac{\pa f}{\pa z_{\alpha\beta}}}{}
 = \sum_{p<q}A_{p\alpha}A_{\beta q}\frac{\pa f}{\pa
 w_{pq}}+\sum_{q<p}A_{q\beta }A_{\alpha p}\frac{\pa f}{\pa w_{pq}}+2\sum_{p=q}A_{p\alpha}A_{\beta q}\frac{\pa f}{\pa
 w_{pq}}\delta_{pq}\nonumber\\
\hphantom{2\ii \frac{\pa f}{\pa z_{\alpha\beta}}}{}
 = \sum_{pq}A_{p\alpha}A_{\beta q}(1+\delta_{pq})\frac{\pa f}{\pa w_{pq}}.\label{for1}
\end{gather}
Similarly, for $\alpha=\beta$, we have
\begin{gather}
2\ii \frac{\pa f}{\pa z_{\alpha\beta}} = \sum_{p\le q}A_{p\alpha}A_{\beta q}\frac{\pa f}{\pa w_{pq}}\nonumber\\
\hphantom{2\ii \frac{\pa f}{\pa z_{\alpha\beta}}}{}
 = \frac{1}{2} \sum_{p< q}A_{p\alpha}A_{\beta q}\frac{\pa f}{\pa
 w_{pq}}+ \frac{1}{2} \sum_{q< p}A_{q\alpha}A_{\beta p}\frac{\pa f}{\pa
 w_{qp}}\frac{1}{2}+ \sum_{p= q}A_{p\alpha}A_{\beta q}\frac{\pa f}{\pa w_{pq}}\nonumber\\
\hphantom{2\ii \frac{\pa f}{\pa z_{\alpha\beta}}}{}
= \frac{1}{2} \sum_{p\not= q}A_{p\alpha}A_{\beta q}\frac{\pa f}{\pa
 w_{pq}}+ \sum_{p= q}A_{p\alpha}A_{\beta q}\frac{\pa f}{\pa
 w_{pq}}= \sum_{p q}A_{p\alpha}A_{\beta q}\frac{1+\delta_{pq}}{2}\frac{\pa f}{\pa
 w_{pq}}.\label{for2}
\end{gather}
We put together \eqref{for1} and \eqref{for2} as in \eqref{DZZ1}.
\end{proof}

With Lemma \ref{masH}, it is obtained the Laplace--Beltrami operator on
the Siegel ball (see \cite[equation~(56)]{maa})
\begin{Proposition}\label{LLB}
The Laplace--Beltrami operator on the Siegel ball $\mc{D}_n$ has the expression
\begin{gather}\label{UndeE}
\Delta_{\mc{D}_n} (W) =\tr \big[N(ND_{\bar{W}})^tD_W\big] 
 =\sum_{p,q,m,n}\db{K}_{pq\bar{m}\bar{n}}\frac{\pa^2}{\pa
 \bar{w}_{mn}\pa w_{pq}}, 
 \qquad N:=\un-W\bar{W},
\\
\label{KPQ} \db{K}_{pq\bar{m}\bar{n}}=e_{mn}e_{pq} K_{pq\bar{m}\bar{n}},\qquad K_{pq\bar{m}\bar{n}} =N_{qn}\bar{N}_{mp},
\end{gather}
which is $\operatorname{Sp}(n,\R)_{\C}$-invariant to the action \eqref{TOIU}. We write down the sum \eqref{UndeE} as
\begin{gather*}
\Delta_{\mc{D}_n}(W)=\sum_{m\leq n;\, p\leq q}k_{mn\bar{p}\bar{q}}(W)\frac{\pa^2}{\pa\bar{w}_{mn}\pa w_{pq}},
\end{gather*}
where, with the expression \eqref{KPQ}, we have
\begin{gather}
k_{pq\bar{m}\bar{n}} =
\frac{1}{4}(K_{pq\bar{m}\bar{n}}+K_{qp\bar{m}\bar{n}}+K_{pq\bar{n}\bar{m}}+K_{qp\bar{n}\bar{m}}),\qquad m<n, \qquad p<q,\nonumber\\
k_{pq\bar{m}\bar{m}} =\frac{1}{2}(K_{pq\bar{m}\bar{m}}+K_{qp\bar{m}\bar{m}}), \qquad p<q, \qquad n=m,\nonumber\\
k_{pp\bar{m}\bar{n}} =\frac{1}{2}(K_{pp\bar{m}\bar{n}}+K_{pp\bar{n}\bar{m}}), \qquad p=q, \qquad m<n,\nonumber\\
k_{pp\bar{m}\bar{m}} = K_{pp\bar{m}\bar{m}}.\label{sp1}
\end{gather}
\eqref{sp1} can be written down as
\begin{gather}
k_{mn\bar{p}\bar{q} } = e_{mn}(1-\delta_{pq}) N_{pm}\bar{N}_{nq} + e_{pq}(1-\delta_{mn}) N_{qm}\bar{N}_{np}
+\delta_{pq}\delta_{mn} N_{pm}\bar{N}_{nq}\nonumber\\
\hphantom{k_{mn\bar{p}\bar{q}}}{} =\frac{1}{2}( N_{qn}\bar{N}_{mp}+ N_{qm}\bar{N}_{np}).\label{ofi}
\end{gather}
\end{Proposition}

\begin{proof}
With equation \eqref{DZZ1} introduced in \eqref{oarecum}, it is obtained \eqref{UndeE}.

In order to directly (i.e., without Remark~\ref{equivv}) prove the invariance of $ \Delta_{\mc{D}_n}$ to the group action $g\times W=W_1$,
$g\in\operatorname{Sp}(n,\R)_C$, $W, W_1\in\mc{D}_n$, we use the formulas \eqref{DW1}, \eqref{DM1}. We write~\eqref{DW1} as
\begin{gather*}\dd w_{pq}= A_{pm}(\dd w_1)_{mn}A_{qn}, \qquad\text{where}\quad A=Wq^*+p^*.
\end{gather*}
With formula \eqref{mircea}, we have successively
\begin{gather}
\frac{\pa w_{pq}}{\pa (w_1)_{\alpha\beta}} = A_{pm}\frac{\pa (w_1)_{mn}}{\pa (w_1)_{\alpha\beta}}A_{nq}
 = A_{pm}A_{nq} \Delta^{mn}_{\alpha\beta}\nonumber\\
\hphantom{\frac{\pa w_{pq}}{\pa (w_1)_{\alpha\beta}}}{}
 =A_{p\alpha}A_{q\beta}+A_{p\beta}A_{q\alpha}-\delta_{\alpha\beta}A_{p\alpha}A_{q\beta}.\label{PARQ}
\end{gather}
With \eqref{PARQ}, we calculate $\frac{\pa f}{\pa (w_1)_{\alpha\beta}}$.

a) Firstly, we consider the case $\alpha\not=\beta$. We have
\begin{gather}
\frac{\pa f}{\pa ( w_1)_{\alpha\beta}}= \left(\sum_{p<q}A_{p\alpha}A_{q\beta} + \sum_{q<p}A_{q\alpha}A_{p\beta}+2\sum_{p=q}A_{p\alpha}A_{q\beta}\right) \frac{\pa f}{\pa w_{pq}}\nonumber\\
\hphantom{\frac{\pa f}{\pa ( w_1)_{\alpha\beta}}}{} =2 \sum A_{p\alpha}A_{q\beta}e_{pq}\frac{\pa f}{\pa w_{pq}}.\label{SUR1}
\end{gather}

b) Now we consider the case $\alpha=\beta$. We have
\begin{gather}\label{SUR2}
\frac{\pa f}{\pa ( w_1)_{\alpha\beta} }= 2f_{\alpha\beta}\sum_{p,q}A_{p\alpha}A_{q\beta}e_{pq}.
\end{gather}
\eqref{SUR1} and \eqref{SUR2} are written together as
\begin{gather}\label{LAB}
D_{W_1}=A^tD_WA.
\end{gather}
Also we write \eqref{DM1} as
\begin{gather}\label{DM2}
M_1=\bar{A}^tMA.
\end{gather}
With \eqref{LAB} and \eqref{DM2} introduced in formula \eqref{UndeE}, we prove
\begin{gather*}\Delta_{\mc{D}_n}(W_1)=\Delta_{\mc{D}_n}(W), \qquad\text{where}\quad W_1=g\times W.\tag*{\qed}
\end{gather*}
\renewcommand{\qed}{}
\end{proof}

\begin{Remark} In the notation of \cite{hua59}
\begin{gather*}Z=\left(\begin{matrix} \sqrt{2}
 z_{11}& z_{12}& \dots& z_{1n}\\
 z_{12}& \sqrt{2}z_{22}&
\dots& z_{2n}\\
\dots&\dots&\dots&\dots\\
 z_{1n}&
 z_{2n}&\dots&\sqrt{2} z_{nn}\end{matrix}\right),\qquad
\pa_Z=\left(\begin{matrix} \sqrt{2}\frac{\pa}{\pa
 z_{11}}&\frac{\pa}{\pa z_{12}}& \dots&\frac{\pa}{\pa z_{1n}}\vspace{1mm}\\
\frac{\pa}{\pa z_{12}}& \sqrt{2}\frac{\pa}{\pa z_{22}}&
\dots&\frac{\pa}{\pa z_{2n}}\\
\dots&\dots&\dots&\dots\\
\frac{\pa}{\pa z_{1n}}&\frac{\pa}{\pa
 z_{2n}}&\dots&\sqrt{2}\frac{\pa}{\pa z_{nn}}\end{matrix}\right),\end{gather*}
the Laplace--Beltrami operator reads (see \cite[equation~(2.6.4)]{hua59})
\begin{gather*}\Delta_{\mc{D}_n}(W)=\tr((\un-\bar{W}W)\pa_W(\un-W\bar{W})\bar{\pa}_W),\end{gather*}
i.e.,
\begin{gather*}
\Delta_{\mc{D}_n}(W) =\sum_{\alpha\beta\lambda\mu=1}^n(\delta_{\lambda\mu}-\sum_{\sigma=1}^nP_{\lambda\sigma}P_{\mu\sigma}w_{\lambda\sigma}\bar{w}_{\mu\sigma})\\
\hphantom{\Delta_{\mc{D}_n}(W) =}{} \times(\delta_{\alpha\beta}-\sum_{\gamma=1}^nP_{\alpha\gamma}P_{\beta\gamma}w_{\alpha\gamma}\bar{w}_{\beta\gamma})
P_{\lambda\alpha}P_{\mu\beta}\frac{\pa^2}{\pa w_{\lambda\alpha}\pa \bar{w}_{\mu\beta}},
 \end{gather*}
where \begin{gather*}P_{\alpha\beta}=1+(\sqrt{2}-1)\delta_{\alpha\beta}.\end{gather*} ``$P$'' is
related with the symbol ``$e$'' in \eqref{symbe} by the relation
\begin{gather*}P_{\alpha\beta} =\sqrt{2}(\sqrt{2}-1)(1+\sqrt{2}e_{\alpha\beta}).\end{gather*}
\end{Remark}

\subsection*{Acknowledgements}
I am grateful to Daniel Beltita for clarifying some aspects of contractibility related with this paper. This research was conducted in the framework of the
ANCS project program PN 16 42 01 01/2016 and UEFISCDI-Romania program PN-II-PCE-55/05.10.2011.

\pdfbookmark[1]{References}{ref}
\LastPageEnding

\end{document}